\documentclass[11pt]{article}
\usepackage{setspace}
\usepackage[hmargin=3cm,vmargin=3cm]{geometry}
\usepackage[numbers]{natbib}
\usepackage{caption}
\usepackage{ae}
\usepackage{aecompl}
\usepackage{subfig}
\usepackage{mathrsfs} 
\usepackage[unicode=true]{hyperref}
\usepackage{breakurl}
\usepackage{afterpage}
\usepackage{floatpag}
\usepackage{booktabs, multicol, multirow}
\usepackage{amssymb}

\usepackage{amsmath, amsthm, amssymb,fullpage}
 \theoremstyle{plain}
  \theoremstyle{definition}
  
  \theoremstyle{plain}

\usepackage{epsfig,psfrag,url,verbatim}
\usepackage{graphicx}
\usepackage{graphicx, algorithm}
\usepackage[T1]{fontenc}
\usepackage{pdfpages}
\usepackage{xargs}
\usepackage{amsfonts}

\usepackage{color}

\newtheorem{cor}{Corollary}

\newtheorem{prop}{Proposition}

\newtheorem{rem}{Remark}

\newcommand{\diag}{\mathrm{diag}}

\newcommand{\B}{\boldsymbol}
\newcommand{\M}{\mathbf}
\newcommand{\R}{\mathrm}

\newcommand{\sbt}{\mathrm{s. t.}}
\newcommand{\rnk}{\mathrm{rank}}
\newcommand{\tr}{\mathrm {tr}}

\newcommand{\p}[1]{\ensuremath{\B\phi^{#1}}}

 \newenvironment{myarray}[2][1]
  {\def\arraystretch{#1}\array{#2}}
  {\endarray}

\DeclareMathOperator*{\argmin}{arg\,min}

\DeclareMathOperator*{\mini}{minimize}

\begin{document}

%\title{Maximum Likelihood Factor Analysis for General Covariance Matrix}
\title{Computation of the Maximum Likelihood estimator in low rank Factor Analysis}

\author{Koulik Khamaru \thanks{Department of Statistics, University of California Berkeley. \texttt{koulik@berkeley.edu}} \and  Rahul Mazumder \thanks{MIT Sloan School of Management, Operations Research Center and Center for Statistics; Massachusetts Institute of Technology. \texttt{rahulmaz@mit.edu}}} 

\maketitle

\abstract{Factor analysis, a classical multivariate statistical technique is popularly used as a fundamental tool for dimensionality reduction in statistics, econometrics and data science. 
Estimation is often carried out via the Maximum Likelihood (ML) principle, which seeks to maximize the likelihood under the assumption that the positive definite covariance matrix can be decomposed as the sum of a low rank positive semidefinite matrix and a diagonal matrix with nonnegative entries. This leads to a challenging rank constrained nonconvex optimization problem. We reformulate the low rank ML Factor Analysis problem as a nonlinear nonsmooth semidefinite optimization problem,
study various structural properties of this reformulation and propose fast and scalable algorithms based on difference of convex (DC) optimization. 
Our approach has computational guarantees, gracefully scales to large problems, is applicable to situations where the sample covariance matrix is rank deficient and adapts to variants of the ML problem with additional constraints on the problem parameters. Our numerical experiments demonstrate the significant usefulness of our approach over existing state-of-the-art approaches.}

%
%\singlespacing
%
%\textcolor{red}{RM: the low rank factor is $\Lambda$, so is the notation for eigenvalues. This is not correct, I missed this point till now.}
%
%\textcolor{red}{RM: the superscripts/subscripts should NOT be bold. I told you to fix this before!!! }
%
%\textcolor{red}{RM: the ijth entry of a matrix should be in small letters.}
%
%\textcolor{red}{RM: please do a spell check. there are many spelling mistakes}

\section{Introduction}\label{sec:intro1}
Factor Analysis (FA)~\citep{anderson2003, Bartholomew_Knott_Moustaki_2011,mardia}, arguably a fundamental dimensionality reduction technique for multivariate data 
has been around for more than a hundred years. FA is popularly used 
to understand the correlation structure among a collection of observed variables in terms of a smaller number of common factors.
In a typical FA model, we assume that the (mean centered) observed random vector $\mathbf{x} \in \Re^{p\times 1}$ may be expressed in the form: $\M{x} = \M{L} \M{f} + \M{u}$,
where, $\M{L} := ((\ell_{ij})) \in \Re^{p \times r}$ is a matrix of factor loadings, $\M{f} \in \Re^{r \times 1}$ is a random vector of scores 
and $\M{u} \in \Re^{p\times 1}$ is a vector of uncorrelated variables.
We assume that $\mathbf{f}$ and $\mathbf{u}$ are uncorrelated, have zero means and 
without loss of generality we set $\mathrm{Cov}(\mathbf{f})= \mathrm{I_r}$---this leads to the following decomposition:  
\begin{equation}\label{fa-decomp-1}
\R{Cov}(\mathbf{x}) :=\B\Sigma = \M{L} \M{L}^\top  +  \M{\Psi},
\end{equation} 
where,  $\mathrm{Cov}(\mathbf{u}):=\B{\Psi} =\diag(\psi_1,\ldots, \psi_p)$.
Decomposition~\eqref{fa-decomp-1} suggests that the population covariance matrix $\B\Sigma:=((\sigma_{ij}))$, can be written as the sum of a low rank positive semidefinite (PSD) matrix and a diagonal matrix $\B\Psi$ with nonnegative entries. In particular, this implies that $\R{Var}(x_i)$ can be decomposed in two parts
$\R{Var}(x_i)= \sigma_{ii}^{2}=  \sum^{r}_{k=1}\ell_{ik}^{2}  + \psi_{i},$ 
where, $\sum^{r}_{k=1} \ell_{ik}^{2}$ represents the variance of $x_i$ which is shared with other variables via common factors (\textit{communality}) and $\psi_{i}$ represents the variance of $x_i$ not shared with other variables (\textit{specific} or \textit{unique variance}).

One of the most popular FA estimation methods is the Maximum Likelihood (ML) procedure \cite{mardia,anderson2003,Bartholomew_Knott_Moustaki_2011}.
Given $n$ multivariate samples $\M{x}_{i}, i = 1, \ldots, n$, which are assumed to be mean-centered, the task is to minimize the negative log-likelihood 
with respect to $\B\Sigma$ that is of the form~\eqref{fa-decomp-1}. This leads to the following optimization problem:
\begin{equation}\label{obj-2-0-0}
\begin{aligned}
\mini~~~& {\mathcal L}(\B\Sigma):=\;\;   - \log\det(\B\Sigma^{-1}) + \R{tr}(\M\Sigma^{-1}\M{S})   \\
\sbt~~~& \M{\Sigma} = \M{\Psi} + \M{L} \M{L}^\top   \\
%& \;\;\; \M{L}\in \Re^{p \times r}  \\
%& \;\;\; \rnk(\M{L})\leq r  \\
& \B{\Psi}=\mathrm{diag}(\psi_1, \ldots, \psi_p) \succeq \epsilon \M{I}, 
\end{aligned}
\end{equation}
where,   $\M{S}=\frac{1}{n} \sum_{i=1}^{n} \M{x}_i \M{x}^\top_i$ is the sample covariance matrix; $\M{\Psi} \in \Re^{p \times p},\M{L} \in \Re^{p \times r}$ and $\B\Sigma$
 are the optimization variables; 
the notation $\M{A} \succeq \M{B}$ means that $\M{A} - \M{B}$ is PSD; and $\tr(\cdot)$ is the usual trace operator.  Here, $\epsilon$ is a small positive constant, specified a-priori to ensure that Problem~\eqref{obj-2-0-0} is bounded below\footnote{Indeed, $\B\Psi \succeq \epsilon \M{I}$ implies that 
$\B\Sigma \succeq \epsilon \M{I} \succ \M{0}$. Thus,  $-\log \det(\B\Sigma^{-1}) \geq p\log(\epsilon)$ and $ \tr \left(\B\Sigma ^{-1} \M{S}\right) \geq 0$ which shows that Problem~\ref{obj-2-0-0} is bounded below. Note that Problem~\eqref{obj-2-0-0} with $\epsilon=0$ need not have a finite solution, i.e., the ML solution need not exist. 
%%We will address in this paper, existence of the ML solution for $\epsilon=0$. However, for the time being we will assume unless otherwise stated that $\epsilon>0$.
 Note that if one of the $\psi_i \rightarrow  +\infty$ then ${\mathcal L}(\B\Sigma) \rightarrow \infty$, a similar argument applies if $\M{L}\M{L}'$ becomes unbounded.
 Thus the infimum of Problem~\eqref{obj-2-0-0} is attained whenever $\epsilon>0$.}. We note that Problem~\eqref{obj-2-0-0} can be rewritten in terms of a new variable 
 $\B\Theta = \M{L}\M{L}' \in \Re^{p \times p}$ and placing a rank constraint on $\B\Theta$ as: $\rnk(\B\Theta) \leq r$ -- hence, in what follows, we will often refer to Problem~\eqref{obj-2-0-0} as a rank constrained
 optimization problem. 
 
% we have $-\log\det \left( \B\Sigma^{-1} \right) = \log\det\left(\B\Sigma \right) \rightarrow +\infty $, $\tr\left( \B\Sigma^{-1} \M{S} \right) \geq 0$ and hence the objective function in Problem~\eqref{obj-2-0-0} approaches to $+\infty$. 
%So the minimum of Problem ~\eqref{obj-2-0-0} is attained at finite $\B\Psi$. Similarly the optimum value $\B\Lambda$ must also be finite. A detailed discussion on the upper bound of the solutions %are provided in Corollary ~\ref{bdd-away-from-zero}.

Observe that Problem~\eqref{obj-2-0-0} is nonconvex in $\B\Sigma$ since (a) the objective function
${\mathcal L}(\B\Sigma)$ is \emph{not} convex~\citep{BV2004} in $\B\Sigma$ and (b) the equality constraint 
$\B\Sigma = \B\Psi + \B{\R{L}} \B{\R{L}}^\top$ is nonconvex.

\paragraph{Related work:} FA  has a long tradition that dates back to more than a hundred years~\cite{spearman1904}. We present here a selective overview 
that is relevant for this work---important contributions in FA have been nicely documented in~\cite{mardia,anderson2003,Bartholomew_Knott_Moustaki_2011,bai-ng-2008large-review}. 
Despite being a problem of fundamental importance in statistical estimation, not much is known about its computational 
properties. Unfortunately, popular off-the-shelf implementations for ML factor analysis (as available in Matlab and R that are in routine use) are quite unstable\footnote{We have observed this in our experiments and this is also reported in our section on numerical experiments.} -- they are based on rather ad-hoc computational algorithms and lead to negative variance estimates which are problematic from a 
statistical inferential viewpoint. 
This is perhaps not surprising, given that the basic problem underlying ML factor analysis is a difficult (nonconvex) 
optimization problem and there has been limited work in developing mathematical optimization based algorithms for this problem.
% Some of the more traditional computational approaches to ML factor analysis are rather ad hoc -- little is known about their 
%convergence properties. They often run into numerical difficulties, and may lead to negative variance estimates. 
It is also difficult to generalize existing algorithms in the presence of additional constraints, depending upon the problem/application context. 
\cite{lawley1971factor,anderson2003} present a nice overview of classical algorithms used for ML factor analysis. 
%%%\cite{rao1955estimation} present an approach based on canonical correlations.
Some modern computational approaches for ML factor analysis are based on the seminal contribution of~\cite{joreskog1967some}. 
This approach requires $\M{S}$ to be of full rank and one needs to assume that $\M{L}$ has exactly rank $r$.
 It applies a (rather ad-hoc) gradient descent like algorithm w.r.t $\B\Psi$.
Recently~\cite{robertson2007maximum} provide necessary and sufficient conditions for existence of a solution of Problem~\eqref{obj-2-0-0} with $\epsilon=0$, however, they do not provide any computational algorithm for solving Problem~\eqref{obj-2-0-0}. 
Another popular approach for ML factor analysis is based on the EM algorithm\cite{rubin1982algorithms,bishop06:_patter_recog_machin_learn}. Publicly available implementations of the EM-type methods apply to the case where, $n$ is smaller than $p$ and hence $\M{S}$ is not full rank.

%Some of the most important contributions to ML methods in FA was made by~\cite{lawley1967,joreskog1967some}, where the authors study Problem~\eqref{obj-2-0-0} when $\M{S}$ is of full rank with 
%$\rnk(\B{\R{L}}) = r$. 
%Another popular line of work is based on the EM framework of~\cite{rubin1982algorithms}. 
%% Recently~\cite{robertson2007maximum} provide necessary and sufficient conditions for existence of a solution of Problem~\eqref{obj-2-0-0} with $\epsilon=0$, 
%%however, they do not provide any computational algorithm for solving Problem~\eqref{obj-2-0-0}, when $\M{S}$ is rank deficient. 
%%Most of the popular algorithms for ML methods are available as R packages (for example: \texttt{`factanal'}, {\texttt{`ifa'}) and Matlab packages.

Not all methods in FA are based on the ML framework.  
In other approaches, one seeks to estimate a matrix $\B\Sigma$ 
of the form $\B\Sigma = \B\Psi + \B{\R{L}}\B{\R{L}}^\top  $, which is close to the sample covariance matrix $\M{S}$ in terms of some 
metric (for e.g., the Frobenius norm: $\| \M{S} - \B\Sigma \|_{F}$). Some popular methods
in the literature are minimum residual FA, principal axis, principal component method, minimum trace FA, among others ---see
%\cite{shapiro-mtfa-82,shapiro-inference-FA-02,berge-91,bai2003inferential,Bartholomew_Knott_Moustaki_2011}
\cite{shapiro-inference-FA-02,Bartholomew_Knott_Moustaki_2011,JMLR:v18:15-613}
 for more description on these approaches.
Fairly recently~\cite{JMLR:v18:15-613} proposed integer optimization based methods for the problem of minimizing
$\| \M{S} - \B\Sigma \|_{F}^2$, where, $\B\Sigma$ is of the form~\eqref{fa-decomp-1} with an additional restriction
of $\M{S} - \B\Sigma \succeq \M{0}$.
 \cite{venkat-2012-factor} study the noiseless decomposition for the FA problem, using nuclear norm relaxations of the rank constraint on $\B\Theta:=\M{L} \M{L}^\top$.
The aforementioned line of work is different from the ML criterion in FA as the data-fidelity measure is different. 
In this paper, our focus is on the computational properties of the ML problem~\eqref{obj-2-0-0}.

\paragraph{Contributions:}
The main contributions of this paper can be summarized as follows:
\begin{enumerate}
\item[1.] We propose a new computational framework for the task of (Gaussian) Maximum Likelihood estimation in Factor Analysis -- a problem central to classical and modern multivariate statistical learning.
The associated optimization problem is challenging: it is given by the minimization of a nonconvex function subject to a rank constraint and additional semidefinite constraints.
 We reformulate the problem as the minimization of a nonsmooth nonconvex function of eigenvalues of a positive semidefinite matrix, subject to (simple) polyhedral constraints.

\item[2.] Using (convexity) properties of spectral functions, we show that the objective function can be expressed as a difference of convex functions; and is hence amenable to computational techniques
in difference of convex optimization (see for example~\cite{hiriart1985generalized}). The computational bottleneck of our algorithm is 
a low rank singular value decomposition (SVD) of a $p \times p$ matrix that needs to be performed for every iteration -- exploiting problem structure, we show that this can be computed 
with cost $O(\min\{n,p\}^2 \max\{n,p\})$.
%%for problems with $p\approx 10^4$ or more (provided $n$ is small) within a few seconds.
%%%which can be computed for quite large problems.
 An important advantage of our proposal, when compared to many other commonly used, publicly available implementations; is that it applies to the case where the sample covariance 
matrix is rank deficient. 
%%%When $n \ll p$, using techniques in numerical linear algebra, we argue that every iteration of our algorithm can be performed with cost $O(n^2p)$.

\item[3.] We explore computational guarantees of our proposed algorithm in terms of reaching a first order stationary point. 
We demonstrate that on a series of numerical examples (including both real and synthetic datasets), our method significantly outperforms commonly used approaches 
for ML factor analysis (in terms of superior numerical stability, smaller computation times and obtaining solutions with better objective values). 
%%%To our knowledge it is the most scalable approach for factor analysis, 
To our knowledge, our proposal is one of the most scalable computational mathematical optimization-based approaches for factor analysis. 
Our approach also generalizes to instances of ML factor analysis where, $\B\Psi$ is not necessarily diagonal. 

%An implementation of our algorithm is available
%as a R-package \textsc{Factmle} on CRAN.

% The special structure of reformulated problem allows us to optimize the problem very easily. The Convex Optimization problem at each iteration of the DC optimization algorithm is solved analytically. This ensures fast convergence with high quality solution. 
%
%\item[3.]The time complexity of the algorithm is low. At each iteration we only need to  calculate the sub-differential of a convex function. This can be calculated by calculating  very few ( $\R{r}$ many ) top  eigenvalues and corresponding eigenvectors of a positive definite matrix. Generally  the rank constraint on $\M{L}$ is much smaller than the dimension of the data (i.e. $\R{r} \ll \R{p} $  ) so each iteration can be computed easily. This shows that our algorithm is scalable.

\end{enumerate}

\paragraph{Notation:} 
%We use the following  notations throughout the paper. Vectors are denoted by a small bold letters  For any $\M{a}_{p\times 1}$, by $ \M{a} \geq 0 \text{ } ( > \M{0}) $ we mean $a_i \geq 0 \text{ }( > 0 )  \forall 1 \leq i \leq p $. \text{ }$\M{a} \geq \M{b} $ means $ \left( \M{a} - \M{b} \right) \geq 0 $. Matrices are denoted by bold capital letters ( e.g. $\M{A} , \M{\Psi}$ etc ). $\M{A}_{ij}$ denote the $(i,j)^{th}$ entry of the matrix $\M{A}$. For a square matrix $\M{A}_{p \times p}$ the characteristic polynomial of $\M{A}$ is denoted by $\chi_{\M{A}}(\lambda):= \det(\lambda\M{I} - \M{A})$. The roots of the characteristic polynomial $\chi_{\M{A}}(\lambda)$ are called the eigenvalues of $\M{A}$ and is denoted as $\lambda (\M{A}):=\left( \lambda_1(\M{A}), \ldots , \lambda_p(\M{A}) \right)$. For a square matrix $\M{A}_{p \times p}$, $\R{diag}(\M{A})$ denote a $p \time 1$ vector formed by diagonal elements of the matrix $\M{A}$. For a square matrix  $\M{A} \succeq \M{0} ( \succ \M{0} )$ means the matrix $\M{A}$ is nonnegative definite ( positive definite ). We define $\M{S}^{*} = \M{\Psi^{-\tfrac{1}{2}} S \Psi^{-\tfrac{1}{2} }}$. The eigenvalues of $\M{S^{*}}$ are $\lambda_1^{*} \geq \lambda_2^{*} \geq \ldots \geq \lambda_p^{*} $. We also define  $\M{L}^{*} = \M{\Psi^{-\tfrac{1}{2}} \Lambda }$, $\M{\Phi}= \M{\Psi^{-1}}$ and $\B{\phi}= \left( \phi_1,\ldots , \phi_p \right)$, with 
%$\B{\Phi} = \R{diag}(\B\phi)$.
For a real symmetric matrix $\M{A}_{p \times p}$, we will denote by $\B\lambda(\M{A})$, the vector of real eigenvalues of $\M{A}$, i.e., 
$(\lambda_{1}(\M{A}), \lambda_{2}(\M{A}), \ldots, \lambda_{p}(\M{A}))$ with $\lambda_{i}(\M{A})\geq \lambda_{i+1}(\M{A})$ for all $i$. For a real postive semidefinite (PSD) matrix 
$\M{A}$ we use $\M{A}^{\frac12}$ to denote its symmetric square root. If $\M{A}$ is invertible and PSD, then we use $\M{A}^{-\frac12}$ to denote the square root 
of $\M{A}^{-1}$. The matrix $\M{I}$ denotes the identity matrix (with dimension determined from the context, unless otherwise mentioned). 
For a vector $\M{a}$, we use the notation $\M{a} \geq \M{0}$ to denote component-wise inequality; for a matrix $\M{A}$, we use 
the notation $\M{A} \succeq \M{0}$ (or $\succ \M{0}$) to denote that the matrix $\M{A}$ is positive semidefinite (respectively, positive definite).
We will assume all diagonals of $\M{S}$ are strictly greater than zero. 
For a non-negative integer $m$ we denote $1, \ldots, m$ by $[m]$.

\section{Methodology} \label{sec:method}

We state a simple result (the proof of which is omitted) regarding the eigenvalues of the product of two matrices -- a property that is used throughout the paper.
\begin{prop}\label{prop-eigAB=eigBA}
For any two symmetric matrices $\M{A}_{p\times p}$ and $\M{B}_{p \times p}$, we have $\B\lambda\left( \M{AB}\right) = \B\lambda \left( \M{BA} \right)$.
\end{prop}
A simple corollary of the above that is used widely in the paper is:
%%% several places of the paper, is stated below.
\begin{cor} \label{corr-eigS*Phi=eigPhi*S}
For any  PSD matrix $\B\Phi$, we have:
%%If $\M{S}^{\tfrac{1}{2}}$ and $\B{\Phi}^{\tfrac{1}{2}}$ denote the symmetric square roots of $\M{S}$ and $\B{\Phi}$ respectively, then
\begin{equation} 
\B\lambda \left( \B{\Phi}^{\tfrac{1}{2}}\M{S}\B{\Phi}^{\tfrac{1}{2}} \right) = \B\lambda \left( \M{S}\B{\Phi} \right) = \B\lambda \left( \B{\Phi} \M{S} \right) = \B\lambda \left( \M{S}^{\tfrac{1}{2}}\B{\Phi}\M{S}^{\tfrac{1}{2}} \right)
\end{equation}
%\begin{proof}
%The proof is an immediate consequence of Proposition~\ref{prop-eigAB=eigBA}.
%\end{proof}
\end{cor}

\subsection{Reformulations}\label{sec:reform}
In this section we present a reformulation of the low rank optimization Problem~\eqref{obj-2-0-0} to one that does not involve any (combinatorial) rank constraint: Proposition~\ref{lem:reform-1} reformulates Problem~\eqref{obj-2-0-0} as an optimization problem in $ \B\Psi $. The resulting problem~\eqref{obj-2-0}
%The final optimization problem mentioned in Corollary~\ref{obj-2-0} 
is amenable to efficient optimization techniques based on difference of convex optimization. 
\textcolor{black}{In Proposition~\ref{bdd-away-from-zero} we provide a characterization of an optimal solution of Problem~\eqref{obj-2-0}.}

\begin{prop}\label{lem:reform-1}
({\bf{a}}) Problem~\eqref{obj-2-0-0} is equivalent to:
\begin{equation}\label{obj-2-0}
\begin{aligned}
 \;\;\;\;\;\;\;  \mini \limits_{}\;\; & 
\displaystyle  \left \{ \log \det(\B{\Psi}) + \R{tr} \left( \M{S}^{*} \right) + \sum^{r}_{i=1} \left( \log(\max \{ 1,\lambda_i^{*}\}) -\max\{1,\lambda_i^{*}\} +1 \right) \right\} \\
\sbt \;\;\;&  \M{\Psi}=\mathrm{diag}(\psi_1, \ldots, \psi_p) \succeq \epsilon \M{I},
\end{aligned}
\end{equation}
where,  $\lambda_1^{*} \geq \lambda_2^{*} \geq \ldots \geq \lambda_p^{*}$ denote the eigenvalues of 
$\M{S}^{*} := \M{\Psi}^{-\tfrac{1}{2}} \M{S} \M{\Psi}^{-\tfrac{1}{2} }$; and the optimization variables are $\B\Psi,\M{S}^*$ and $\{\lambda^*_{i}\}_{i \geq 1}$.

({\bf{b}}) Suppose $\hat{\B{\Psi}}$ is a minimizer of Problem~\eqref{obj-2-0} and $\hat{\M{L}} =  {\hat{\B{\Psi}}}^{\tfrac{1}{2}} \left[ \hat{\M{z}}_1,  \hat{\M{z}}_2 ,  \ldots , \hat{\M{z}}_r \right]$
where, $  \hat{\M{z}}_1,  \hat{\M{z}}_2 ,  \ldots ,  \hat{\M{z}}_r $  are eigenvectors of
$\M{\hat{\Psi}}^{-\tfrac{1}{2}} \M{S} \M{\hat{\Psi}}^{-\tfrac{1}{2} }$ 
corresponding to its top $r$ eigenvalues
$\hat{\lambda}_i^{*}, i \leq r$ with $\|  \hat{\M{z}}_i \|^{2} = \max \left\{ 1,\hat{\lambda}_i^{*} \right\} -  1$ for all $i=1, \ldots, r$. Then 
$(\hat{\B\Psi}, \hat{\M{L}})$ is a minimizer of Problem~\eqref{obj-2-0-0}.
%\textcolor{red}{RM: is the above equation required??}

\begin{proof}
Part (a): We first minimize Problem~\eqref{obj-2-0-0} with respect to $\M{L}$ for a fixed value of $\M{\Psi}$. 
%We start by observing that $\M{\Sigma}$ is in the form $\M{A + BB'}$. Woobery identity gives us a nice formula for $\M{\Sigma^{-1}}$
A simple application of the Sherman Woodbury formula, with some rearrangement gives:
%\begin{eqnarray} \label{siginv}
%(\M{\Psi + \Lambda \Lambda')^{-1}} & = & \M{\Psi ^{-1} - \Psi ^{-1} \Lambda(I + \Lambda'\Psi ^{-1} \Lambda)^{-1} \Lambda' \Psi ^{-1}} \nonumber \\
%&=& \B\Psi^{-1} \left( \M{I} - \M{L}(\M{I} + \M{L}'\B\Psi^{-1} \M{L})^{-1} \M{L}' \B\Psi ^{-1} \right).
%\end{eqnarray}
\begin{equation}\label{sherman-woodbury-1}
\begin{myarray}[1.1]{ r c l}
\B\Sigma^{-1} &=& \left(\B{\Psi} + \M{L} \M{L}^\top \right)^{-1} \\
& =& \B\Psi^{-1} - \B\Psi^{-1} \M{L} \left(\M{I} + \M{L}^\top\B\Psi^{-1} \M{L}\right)^{-1} \M{L}^\top \B\Psi^{-1}\\
%%&&=&  \B\Psi^{-1} \left( \M{I} - \M{L}(\M{I} + \M{L}^\top\B\Psi^{-1} \M{L})^{-1} \M{L}^\top \B\Psi^{-1} \right).  \\
& =& \B\Psi^{-1} -  \B\Psi^{-\frac12}\B\Psi^{-\frac12}\M{L}(\M{I} + \M{L}^\top\B\Psi^{-\frac12}\B\Psi^{-\frac12} \M{L})^{-1} \M{L}^\top \B\Psi^{-\frac12} \B\Psi^{-\frac12}
\end{myarray}
\end{equation}
Writing $\M{L}^{*} = \M{\Psi}^{-\tfrac{1}{2}} \M{L}$ in the last line of display~\eqref{sherman-woodbury-1}, we get:
$$ \B\Sigma^{-1} = \B\Psi^{-1} -  \B\Psi^{-\frac12}  \M{L}^{*} (\M{I} + (\M{L}^{*})^\top\M{L}^{*})^{-1} (\M{L}^{*})^\top \B\Psi^{-\frac12}. $$

The above implies that:
\begin{equation}\label{line-prop2-2}
\begin{aligned}
\tr(\B\Sigma^{-1}\M{S}) =& \tr(\B\Psi^{-1} \M{S}) - \tr \left( \B\Psi^{-\frac12}  \M{L}^{*} \left(\M{I} + (\M{L}^{*})^\top\M{L}^{*} \right)^{-1} (\M{L}^{*})^\top \B\Psi^{-\frac12} \M{S} \right) \\
=& \tr(\B\Psi^{-\frac12} \M{S}\B\Psi^{-\frac12}) - \tr \left( (\M{L}^{*})^\top \B\Psi^{-\frac12}\M{S} \B\Psi^{-\frac12}  \M{L}^{*} \left(\M{I} + (\M{L}^{*})^\top\M{L}^{*} \right)^{-1}   \right) \\
=&  \tr(\M{S}^*) - \tr \left((\M{L}^{*})^\top \M{S}^*\M{L}^{*}  \left(\M{I} + (\M{L}^{*})^\top\M{L}^{*} \right)^{-1} \right)~~~~~\text{(Using, $\M{S}^* = \B\Psi^{-\frac12} \M{S}\B\Psi^{-\frac12}$)}
\end{aligned}
\end{equation}
where, in the second and third lines of display~\eqref{line-prop2-2} we (repeatedly) used the fact that $\tr(AB) = \tr(BA)$. 

%Note that, $\B\Psi + \M{L}^\top\M{L} = \B\Psi^{\frac{12}}(  \M{I}  + \B\Psi^{-\frac12}\M{L}^\top\M{L} \B\Psi^{-\frac12} )  \B\Psi^{\frac{12}}$
In addition, note that:
\begin{equation}\label{line-prop2-3}
\begin{myarray}[1.1]{r c l}
-\log\det(\B\Sigma^{-1}) = \log\det(\B\Sigma) &=& \log\det(\B\Psi + \M{L} \M{L}^\top) \\
&=& \log\det(\B\Psi) + \log\det\left(\M{I} + {\M{L}^{*}}^\top \M{L}^* \right) 
\end{myarray}
\end{equation}
%%%Using~\eqref{line-prop2-2},~\eqref{line-prop2-3} we have that:
%%%Note that  
Using~\eqref{line-prop2-2},~\eqref{line-prop2-3} in the objective function of Problem~\eqref{obj-2-0-0}, i.e., 
${\mathcal L}(\B\Sigma) = -\log\det(\B\Sigma^{-1}) + \tr(\B\Sigma^{-1}\M{S})$ we get the following equivalent reformulation of Problem~\eqref{obj-2-0-0}:
%%%%$$ {\mathcal L}(\B\Sigma) =  \log \det(\M{\Psi}) + \R{tr} \left( \M{S}^{*} \right) + \log\det\left(\M{I} + {\M{L}^{*}}^\top \M{L}^* \right) \\& &  \qquad \;\;\;\;\;\;\;\;\;\;\  - \R{tr}\left({\M{L}^{*}}^\top {\M{S}^*} \M{L}^* \left(\M{I} + {\M{L}^{*}}^\top \M{L}^*\right)^{-1}\right)$$
\begin{equation}\label{eqn2-0-1}
\begin{myarray}[1.1]{c l l}
\mini~~& h(\M{\Psi},\M{L}):=& \log \det(\M{\Psi}) + \log\det\left(\M{I} + {\M{L}^{*}}^\top \M{L}^* \right) + \R{tr} \left( \M{S}^{*} \right) \\
&                                            &  - \R{tr}\left({\M{L}^{*}}^\top {\M{S}^*} \M{L}^* \left(\M{I} + {\M{L}^{*}}^\top \M{L}^*\right)^{-1}\right) \\
\sbt~~&& \M{\Psi}=\mathrm{diag}(\psi_1, \ldots, \psi_p) \succeq \epsilon \M{I},
%& & \rnk(\M{L}) \leq \R{r}
\end{myarray}
\end{equation}
where, recall we use the notation: $\M{L}^{*} = \M{\Psi}^{-\tfrac{1}{2}} \M{L} $ and $\M{S}^{*} = \M{\Psi}^{-\tfrac{1}{2}} \M{S} \M{\Psi}^{-\tfrac{1}{2} }$; and the optimization variables in 
Problem~\eqref{eqn2-0-1} are $\B\Psi, \M{L}$ (and consequently, $\M{L}^*, \M{S}^{*}$).
Note that $h\left(\M{\Psi},\M{L}\right) = h(\M{\Psi},\M{L} \M{U}) $ for any orthogonal matrix $\M{U}$. So we can  substitute  $\M{L}$ by $ \M{L} \M{U }$ 
in Problem~\eqref{eqn2-0-1}. We choose $\M{U}$ such that the columns of $\M{L}^*$ are orthogonal or zero vectors. 
%Now we differentiate the log-likelihood in \eqref{eqn2-0-1} w.r.t. $\M{L}$ --- this leads to:
% Lawley and Maxwell ~\cite{lawley1940,lawley1967} gave a convenient form for the derivative. They showed 
Note that the derivative of $h(\M{\Psi},\M{L})= \log\det(\M{L}\M{L}^\top + \B\Psi) + \tr((\M{L}\M{L}^\top + \B\Psi)^{-1} \M{S})$ 
w.r.t. $\M{L}$ is given by:
\begin{equation}
\frac{\partial h(\M{\Psi},\M{L}) }{\partial \M{L}} = -2\B{\Sigma}^{-1}\left( \M{\Sigma} - \M{S} \right) \B{\Sigma}^{-1}\M{L}. \label{stationary-eqn-lambda-1}
\end{equation}
%If in addition $\dfrac{\partial h(\M{\Psi},\M{L}) }{\partial \M{L}}  = \M{0}$, then a little algebra shows that 
Setting the above gradient to be zero and writing $\B\Sigma = \B\Psi + \M{L} \M{L}^\top$ we have $\M{L} = \M{S}( \B\Psi + \M{L} \M{L}^\top)^{-1}\M{L}$.
After some elementary algebra this can be written as:
\begin{equation}
\M{S}^{*}\M{L}^{*}\left( \M{I} + {\M{L}^{*^\top}}\M{L}^*  \right)^{-1} = \M{L}^{*} . \label{stationary-eqn-lambda-2}
\end{equation}
Since we choose the columns of $\M{L}^*$ to be pairwise orthogonal or zero vectors, it follows that $\M{I} + {\M{L}^{*\top}} \M{L}^*$ is a diagonal matrix. This means that \eqref{stationary-eqn-lambda-2} is a collection of eigenvector equations for the matrix $\M{S}^*$. From condition~\eqref{stationary-eqn-lambda-2} we have that at an optimal value of $\M{L}$, the columns of $\M{L}^*$ are either pairwise orthogonal eigenvectors of $\M{S^{*}}$ with eigenvalues as the diagonal entries of $ \M{I} + \M{L}^{*\top}{\M{L}^*}$  or they are zero vectors. Denoting the columns of $\M{L}^*$ by $\M{z}_i,i=1,\ldots,r $ the part of the function $h(\M{\Psi},\M{L})$ in display~\eqref{eqn2-0-1}, that depends upon $\M{L}$ is given by:
\begin{equation} \label{eqn2-0-2}
g(\M{L})=\sum_{i=1}^{r} \left(\log(1 + \M{z}_i^\top \M{z}_i) - \frac{\M{z}_i^\top \M{S}^* \M{z}_i}{1 + \M{z}_i^\top \M{z}_i} \right).
\end{equation}
%where, above, we have ignored terms free of $\M{L}$. 
Since ${\M{z}_i}$s are pairwise orthogonal or zero vectors, it follows from equation \eqref{stationary-eqn-lambda-2} that
\begin{equation}
\M{S}^*\M{z}_i=\beta_i \M{z}_i , \text{ with } \beta_i=1+\M{z}_i^\top \M{z}_i . \label{stationary-eqn-lambda-simple}
\end{equation}
Note that in the above equation, either $\beta_i =1 $ with $\M{z}_i = 0 $ or $\beta_i > 1$ and $\beta_i $ equals some eigenvalue of $\M{S}^{*}$ with eigenvector $\M{z}_i$ --- 
thus~\eqref{eqn2-0-2} becomes
 \begin{equation} \label{eqn2-0-3}
 g(\M{L})=\sum_{i=1}^{r} \left(\log(\beta_i) - \beta_i +1 \right).
 \end{equation}
Note that $\beta \mapsto \log(\beta) - \beta +1 $ is strictly decreasing for all $\beta \geq 1 $. So it is easy to see that \eqref{eqn2-0-3} is minimized for $\beta_i = \max \{ 1,\lambda_i^{*} \}$ for 
$1\leq i \leq r,$ where, $\lambda_1^{*} \geq \ldots \geq \lambda_r^{*} $ are the top $r$ eigenvalues of $\M{S}^{*}$. 
The optimal choice of $\M{z}_i$ is given by  $\M{z}_i = 0$ when $\beta_i=1$ and when $\beta_i > 1$, $\M{z}_i$ is
 an eigenvector of $\M{S}^{*}$ with eigenvalue $\lambda_i^{*}$ and with $\M{z}_i^\top \M{z}_i  = \max
  \{ 1, \lambda_i^{*} \} -1 $.
Noting that 
%\begin{equation}
%\begin{aligned}
%\min_{\B\Phi, \M{L}}& h(\B\Phi, \M{L}) &&=&&  \min_{\B\Phi}& \;\; , \min_{\M{L}} \; \; h(\B\Phi, \M{L}) \\
%\sbt & \M{\Psi}=\mathrm{diag}(\psi_1, \ldots, \psi_p) \succeq \epsilon \M{I},&&=&& \sbt & \M{\Psi}=\mathrm{diag}(\psi_1, \ldots, \psi_p) \succeq \epsilon \M{I},
%\end{aligned}
%\end{equation}
\begin{equation}
\begin{aligned}
\min_{\B\Psi \succeq \epsilon \M{I}, \M{L}}~~& h(\B\Psi, \M{L}) &&=&&  \min_{\B\Psi  \succeq \epsilon \M{I}}& \left\{\min_{\M{L}} \; \; h(\B\Psi, \M{L})\right\},
\end{aligned}
\end{equation}
and substituting the values of $\M{L}$ that minimize the inner minimization problem, above, into the objective function 
$h(\B\Psi, \M{L})$, we obtain~\eqref{obj-2-0}. 
%\begin{equation}  \label{eqn2-0-4}
%\log\det(\M{\Psi}) + \R{tr} \left( \M{S^{*}} \right) + \sum^{r}_{i=1} \left( \log(\R{max}\{1,\lambda_i^{*}\}) -\R{max}\{1,\lambda_i^{*}\} +1 \right) 
%\end{equation}
%Note that equation ~\eqref{eqn:neglog-likehdd-bdd} shows that Problem~\eqref{obj-2-0-0} is bounded below, so solution to Problem~\eqref{obj-2-0-0} or equivalently Problem~\eqref{obj-2-0} always exists. 

Part (b): The proof of this part is a consequence of the proof of Part (a).
%If $\M{\hat{\Psi}}$ is a minimizer of Problem~\eqref{obj-2-0} then the above proposition enables us to obtain the optimal value of $\M{\hat{L}}$ and $(\M{\hat{\Psi},\hat{L}})$ minimizes Problem~\eqref{obj-2-0-0}.
%%This completes the proof of the proposition.
\end{proof}
\end{prop}

%of maximizing the log-likelihood function w.r.t. $\B{\Lambda}$ keeping $\B{\Psi}$ fixed is not entirely new and this type of methods are known as concentrating out $\B{\Lambda}$. This type of work was first introduced by Lawley, Maxwell and Joreskog ~\cite{lawley1940,lawley1967,joreskog1967some}. But in their work they assumed that the Sample covariance matrix $\M{S}$ is of full rank. Recently Robertson ~\cite{robertson2007maximum} developed a method for concentrating out $\B{\Lambda}$ with out any assumption on the Sample covariance matrix $\M{S}$, but their goal was to discuss the existance of maximum likelihood solution of Factor model. Although there is a similarity in the method used in ~\cite{robertson2007maximum} and that used in the above proposition, the final expression \eqref{obj-2-0} is quite different from that in ~\cite{robertson2007maximum}. The expression in \eqref{obj-2-0} is simpler and easy to work with.
The method of minimizing the objective function w.r.t. $\M{L}$ with $\B\Psi$ held fixed, is inspired by the classical 
work of~\cite{lawley1940,lawley1967,joreskog1967some}---this line of work however, assumes $\M{S}$ to be of full rank. 
We do not assume that $\M{S}$ is full rank. 
\cite{robertson2007maximum} investigate the existence of ML solutions 
for a general $\M{S}$ --- no algorithms for computing the solution are presented. 
The expression~\eqref{obj-2-0} derived herein, does not appear in~\cite{robertson2007maximum}. Formulation~\eqref{obj-2-0} plays a key role in
developing algorithms for Problem~\eqref{obj-2-0}, a main focus of this paper.

Proposition~\ref{bdd-away-from-zero} shows that any solution of Problem~\eqref{obj-2-0} is bounded above.
%\textcolor{red}{Koulik: Start copying from here}

\begin{prop} \label{bdd-away-from-zero}
Let $\B{\hat{\Psi}}$ be a solution of Problem~\eqref{obj-2-0}. Then $\diag(\M{S}) \succeq \hat{\B{\Psi}} \succeq \epsilon \M{I}$.
\begin{proof}
Note that for any $\B\Psi \succ \M{0}$,   $\frac{\partial h }{\partial \B{\Psi}} = \diag \left( \B\Sigma ^{-1} \left(\B\Sigma -  \M{S}\right) \B\Sigma ^{-1} \right)$. Let $\M{L}\left( \B\Psi \right)$ be the optimal value of $\M{L}$ for fixed value of $\B\Psi$. Then from the proof of Proposition~\ref{lem:reform-1}, it follows that $\frac{\partial h}{\partial \M{L}} = 0$ at $\left( \B\Psi,\M{L} \left( \B\Psi \right) \right)$. Setting~\eqref{stationary-eqn-lambda-1} to zero, we have $\M{L}=\M{S}\B{\Sigma}^{-1}\M{L}$; and applying Sherman Woodbury formula
on $(\B\Psi + \M{L}\M{L}^\top)^{-1}$, we have the following chain of inequalities:
\begin{eqnarray}
\M{L}&=&\M{S}\B{\Sigma}^{-1}\M{L} \nonumber \\
&=& \M{S}\left( \B\Psi^{-1} - \B{\Psi}^{-1}\M{L}\left( \M{I} + \M{L}^\top \M{\Psi}^{-1}\M{L} \right)^{-1}\M{L}^\top \M{\Psi}^{-1}   \right)\M{L}\label{eqn:line-0-1}\\ 
&=& \M{S}\B\Psi^{-1}\M{L} - \M{S}\B{\Psi}^{-1}\M{L} \left( \left( \M{I} + \M{L}^\top \M{\Psi}^{-1}\M{L} \right)^{-1} \M{L}^\top \M{\Psi}^{-1}\M{L} \right)\label{eqn:line-0-2} \\
&=& \M{S}\B\Psi^{-1}\M{L} - \M{S}\B{\Psi}^{-1}\M{L} \left( \M{I} - \left( \M{I} + \M{L}^\top \M{\Psi}^{-1}\M{L} \right)^{-1}  \right)\label{eqn:line-0-3}\\
&=& \M{S}\B{\Psi}^{-1}\M{L}\left( \M{I} + \M{L}^\top \M{\Psi}^{-1}\M{L} \right)^{-1}. \label{eqn:line-0-4}
\end{eqnarray}
Eqn~\eqref{eqn:line-0-1} follows from~\eqref{sherman-woodbury-1}; Eqn~\eqref{eqn:line-0-3} follows from~\eqref{eqn:line-0-2} by using the observation that for a PSD matrix $\M{B}$ we have the 
following identity: $(\M{I}+\M{B})^{-1}\M{B} = \M{I} - (\M{I} + \M{B})^{-1}$ (which can be verified by simple algebra).

Moreover using~\eqref{sherman-woodbury-1} on $\B\Sigma^{-1}$ the expression $\M{S}\M{\Sigma}^{-1}$ simplifies as: 
\begin{eqnarray}
\M{S}\M{\Sigma}^{-1} &=& \M{S}\M{\Psi}^{-1} - \left(\M{S}\B{\Psi}^{-1}\M{L}\left( \M{I} + \M{L}^\top \M{\Psi}^{-1}\M{L} \right)^{-1} \right) \left( \M{L}^\top\M{\Psi}^{-1}\right) \nonumber \\
&=& \M{S}\M{\Psi}^{-1} - \M{L}\M{L}^\top\M{\Psi}^{-1} \nonumber ~~~ \text{(Using expression of $\M{L}$ from rhs of~\eqref{eqn:line-0-4})}  \nonumber \\
&=& \M{S}\M{\Psi}^{-1} - \left( \M{\Sigma} - \B\Psi \right) \M{\Psi}^{-1} ~~~\text{(Since, $\B\Sigma = \M{L}\M{L}^\top +\B\Psi$)}\nonumber \\
&=& \M{S}\M{\Psi}^{-1} - \M{\Sigma}\M{\Psi}^{-1} + \M{I}.  \label{eqn: S-Sigma-inv}
\end{eqnarray}
%%%Using the expression for $ \M{S}\M{\Sigma}^{-1} $ and $ \M{\Sigma}^{-1}\M{S}  = \left( \M{S}\M{\Sigma}^{-1} \right)^\top $ from above we have a simplified expression for 
Note that we have the following expression for $\B\Sigma ^{-1} \left(\B\Sigma -  \M{S}\right) \B\Sigma ^{-1}$: 
\begin{equation}\label{line-1-1-1-1}
\begin{myarray}[1.1]{r l}
\B\Sigma ^{-1} \left(\B\Sigma -  \M{S}\right) \B\Sigma ^{-1} =& \B\Sigma ^{-1}  - \B\Sigma ^{-1} (\M{S} \B\Sigma^{-1})\\
=& \B\Sigma ^{-1} - \B\Sigma ^{-1}\left( \M{S}\M{\Psi}^{-1} - \M{\Sigma}\M{\Psi}^{-1} + \M{I}  \right) \\
=& - \left(\B\Sigma ^{-1} \M{S} \right)\M{\Psi}^{-1} + \M{\Psi}^{-1} \\
=& - \left( \M{\Psi}^{-1}\M{S} -\M{\Psi}^{-1}\B\Sigma +\M{I} \right)\M{\Psi}^{-1} + \M{\Psi}^{-1} \\
=& \B\Psi ^{-1} \left(\B\Sigma -  \M{S}\right) \B\Psi ^{-1}.
\end{myarray}
\end{equation}
where, the second line follows by using expression of $\M{S}\M{\Sigma}^{-1}$ from~\eqref{eqn: S-Sigma-inv}; and 
the fourth line follows by using the same expression for $\M{\Sigma}^{-1}\M{S}  = (\M{S}\M{\Sigma}^{-1})^\top$.

Using~\eqref{line-1-1-1-1}, the expression for $\frac{\partial h}{\partial \B{\Psi}}=\B\Sigma ^{-1} \left(\B\Sigma -  \M{S}\right) \B\Sigma ^{-1}$ reduces to
\begin{equation}\label{stationary-psi-upd}
 \frac{\partial h }{\partial \B{\Psi}} = \diag \left( \B\Sigma ^{-1} \left(\B\Sigma -  \M{S}\right) \B\Sigma ^{-1} \right) =\B\Psi^{-1} \diag \left(\B\Sigma -  \M{S}\right)  \B\Psi^{-1}.
\end{equation}
So, we conclude that $\frac{\partial h}{\partial \psi_i} > 0$ if $\psi_i > s_{ii}$ for some $1 \leq i \leq p$. We are minimizing Problem ~\eqref{obj-2-0}, so the optimal solution must satisfy $\hat{\psi}_i \leq s_{ii}$. This completes the proof.
\end{proof}
%\textcolor{red}{Koulik: Copy upto this.}
\end{prop}
%%The above proposition implies that a minimum of Problem~\eqref{obj-2-0-0} exists.

%%\textcolor{red}{RM: Please be more careful. I am losing my patience. Why $ \B{\hat{\Psi}} \succeq \frac{1}{\epsilon} \M{I} $????}
%%
%%\textcolor{red}{RM: The proof should be more elaborate. I cannot follow.}
%%
%%\textcolor{red}{RM: WHY is there a $\mathcal{L}$?? it should be $h()$.}
%%
%%\textcolor{red}{RM: WHY is the "hat" bold?? I told you to fix this, but you did not}
%%
%%
%%\textcolor{red}{add statement on feasibility and existence (trivial)}

We present another equivalent representation of Problem~\eqref{obj-2-0}, by a simple change of variables $\B{\Phi}:= \B\Psi^{-1}$. In what follows below, unless otherwise specified, we use the shorthand  $\B{\Phi}:= \R{diag}\left( \phi_1, \ldots ,\phi_p \right)$ and $\B{\phi} = \left( \phi_1, \ldots ,\phi_p \right)$.

\begin{cor}\label{cor-obj-2-0}
Problem~\eqref{obj-2-0} is equivalent to the following optimization problem in $\B\phi$:
%\begin{equation}\label{obj-2}
%\begin{aligned}
%\mini~~~& 
%\displaystyle f(\B{\phi}):= \sum_{i=1}^{p} \left(-\log \phi_i +   s_{ii}\phi_i \right) \\
%&~~~~~~~~~+ \sum^{r}_{i=1} \left( \log\left(\max\{1,\lambda_i^{*}\}\right) - \max \{1,\lambda_i^{*}\} +1 \right) \\
%\sbt~~~& \M{0} \prec  \B{\Phi}=\mathrm{diag}(\phi_1, \ldots, \phi_p) \preceq  \frac{1}{\epsilon} \M{I}.
%\end{aligned}
%\end{equation}
\begin{equation}\label{obj-2}
\begin{aligned}
\mini~~~& 
\displaystyle f(\B{\phi}):= \sum_{i=1}^{p} \left(-\log \phi_i +   s_{ii}\phi_i \right) + \sum^{r}_{i=1} \left( \log\left(\max\{1,\lambda_i^{*}\}\right) - \max \{1,\lambda_i^{*}\} +1 \right) \\
\sbt~~~& \M{0} \prec  \B{\Phi}=\mathrm{diag}(\phi_1, \ldots, \phi_p) \preceq  \frac{1}{\epsilon} \M{I},
\end{aligned}
\end{equation}
where, $\{{\lambda}^*\}_1^{r}$ are the top $r$ eigenvalues of $\M{S}^{*} = \M{\Phi}^{\tfrac{1}{2}} \M{S} \M{\Phi}^{\tfrac{1}{2}}$. If $\hat{\B\Psi}$ is a solution to Problem~\eqref{obj-2-0} then, we have
$ \hat{\B\Psi} = \hat{\B\Phi}^{-1},$ where, $\hat{\B\Phi}$ is a solution to Problem~\eqref{obj-2}.
%\begin{proof}
%Observe that $\log\det(\B\Psi) = - \sum_{i} \log(\phi_{i})$ and $\tr( \M{S}^*) = \sum_{i} s_{ii}\phi_{i}$. Making these substitutions in Problem~\eqref{obj-2-0}, we arrive at~\eqref{obj-2}. The second part follows from Corollary~\ref{bdd-away-from-zero}.
%%\begin{align}
%%&\phantom{=} -\log |\M{\Phi}| + \R{tr} \left( \M{S^{*}} \right) + \sum^{r}_{i=1} \left[ \log(\R{max}\{1,\lambda_i^{*}\}) -\R{max}\{1,\lambda_i^{*}\} +1 \right] \nonumber \\
%%&=  -\sum_{i=1}^{p}\log \phi_i + \R{tr} \left( \M{S \Phi} \right) + \sum^{r}_{i=1} \left[ \log(\R{max}\{1,\lambda_i^{*}\}) -\R{max}\{1,\lambda_i^{*}\} +1 \right] \nonumber \\
%%&= \sum_{i=1}^{p} [ -\log \phi_i +   \M{S}_{ii}\phi_i ] + \sum^{r}_{i=1} \left[ \log(\R{max}\{1,\lambda_i^{*}\}) -\R{max}\{1,\lambda_i^{*}\} +1 \right] \label{eqn2-1}
%%\end{align}
%%Where $\M{\Phi} = \M{\Psi^{-1}}$ and $\R{diag}(\M{\Phi})=\B{\phi}= (\phi_1,\phi_2,\ldots,\phi_p)$. Now using~\eqref{eqn2-1} we can rewrite 
%\end{proof}
\end{cor}

%\textcolor{red}{why will a solution to the above problem exist. Write one sentence on this.}

\begin{rem} Problem~\eqref{obj-2-0} (and Problem~\eqref{obj-2}) is a minimization problem in $\B\Psi$ (respectively,  $\M{\Phi}$), 
unlike the (rank constrained) Problem~\eqref{obj-2-0-0} with variables 
$\M{L}$ and $\M\Psi$.
%%% Note that the constraints of Problem~\eqref{obj-2} is convex;
Note that Problem~\eqref{obj-2} is  nonconvex due to the nonconvex objective function, though the  constraints are convex.
Corollary~\ref{corr:diff-cvx1} shows that the the objective function $f(\B\phi)$ appearing in Problem~\eqref{obj-2} is neither convex nor concave but it can be written as a difference of simple 
convex functions.
\end{rem}

\subsection{Expressing Problem~\eqref{obj-2} as a difference of convex functions}
In this section we show via Proposition~\ref{lem-diff-convex1-0} and ~\ref{lem-diff-convex1} that the objective function in Problem~\eqref{obj-2} can be written as a difference of two convex functions (Corollary~\ref{corr:diff-cvx1}). This renders the application of algorithms based on difference of convex optimization, to get good solutions to Problem~\eqref{obj-2}. 
%%%we reformulate problem~\ref{obj-2} as a difference of convex (DC) optimization problem. 
Proposition~\ref{Spectral function} shows that when the sample covariance matrix $ \M{S} $ is full rank, the objective function in Problem~\eqref{obj-2} can be expressed
purely in terms of the eigenvalues of $\M{S}^*$.

%In this section we will discuss some properties of Problem~\eqref{obj-2}. 
%First we will show the objective function of~\eqref{obj-2} can be written as a difference of convex functions.
%Consider the first part of the objective function: 
%\begin{equation}\label{log-linear-part}
%\B{\phi}:= (\phi_{1}, \ldots, \phi_{p}) \mapsto \sum_{i=1}^{p} \left(-\log \phi_i +   \M{S}_{ii}\phi_i \right).
%\end{equation}  
%The 
%function~\eqref{log-linear-part} is convex in $\B{\phi}$, due to the convexity of the map $\phi_{i} \mapsto -\log\phi_{i}$ and the linearity of
%the second term. 
%The second part of the objective function in \eqref{obj-2}, is of the following form:

Let  $y_{(1)} \geq \ldots \geq y_{(p)}$ be an ordering of $\{y_{i}\}_{1}^{p} \in [0,\infty)^{p}$ and define:
\begin{equation}\label{top-sum11}
H_{r}(\M{y}):= \sum_{i=1}^{r} \left( \log\left(\max \{ 1,y_{(i)} \} \right) -\max \{1,y_{(i)}\} +1 \right).
\end{equation}
The following proposition shows that $\M{y} \mapsto H_{r}(\M{y})$ is concave on $\M{y} \geq \M{0}$.

\begin{prop}\label{lem-diff-convex1-0}
 $H_{r}(\M{y})$ as defined in~\eqref{top-sum11}, is concave on 
$\M{y} \geq \M{0}$.
\begin{proof}
We first establish that $H_{r}(\M{y})$ admits the following representation:
\begin{equation}\label{point-wise-inf1}
\begin{myarray}[1.1]{l c l}
H_{r}(\M{y}) =& \min\limits_{\M{w}} & \widetilde{H}(\M{w};\M{y}):= \sum\limits_{i=1}^{p}w_{i} \left ( \log\left(\max\{1,y_{i} \}\right) - \max \{1,y_{i}\} +1 \right) \\
&\sbt  &  \sum\limits_{i=1}^{p} w_{i} = r, ~~~ 0 \leq  w_i \leq 1, i \in  [p],
\end{myarray}
\end{equation}
as the minimum (w.r.t. the optimization variable $\M{w}$) of the linear functional $\M{w} \mapsto \widetilde{H}(\M{w};\M{y})$.
To see why this is true, note that the scalar function 
$y \mapsto  \log\left(\R{max}\{1,y \}\right) -\R{max} \{1,y\} +1$ is decreasing on $y \geq 0$.
%Hence, we have:
%$$ \log\left(\R{max}\{1,y_{(i)} \}\right) -\R{max} \{1,y_{(i)}\} +1   \leq   \log\left(\R{max}\{1,y_{(i+1)} \}\right) -\R{max} \{1,y_{(i+1)}\} +1,$$ 
%for $i = 1, \ldots, p-1$.
Hence the sum 
$\sum_{i=1}^{p}w_{i} \left ( \log\left(\R{max}\{1,y_{i} \}\right) -\R{max} \{1,y_{i}\} +1 \right) $ will be minimized 
for a choice: $w_{i} = 1$ whenever  $y_{i}$ is one of the top $r$ elements among $y_{1}, \ldots, y_{p}$; and $w_{i}=0$ for all other choices of $i$.
This justifies representation~\eqref{point-wise-inf1}.

For any $\Re \ni y \geq 0$, note that $y \mapsto \log\left(\R{max}\{1,y\}\right) -\R{max}\{1,y\} +1 $ is \emph{concave}. 
Hence, for every fixed $\Re^{p} \ni \M{w} \geq \M{0}$, the function 
$$(y_{1}, \ldots, y_{p}) \mapsto \sum_{i=1}^{p}w_{i}\left( \log\left(\R{max}\{1,y_{i}\}\right) -\R{max}\{1,y_{i}\} +1 \right)  $$ is concave 
on $\M{y} \geq \M{0}$. Since the pointwise infimum of a family of concave functions is concave~\cite{BV2004}, $H_{r}(\M{y})$ is concave on the nonnegative reals. 
\end{proof}
\end{prop}

\begin{prop}\label{lem-diff-convex1}
For any  $r \in [p]$,  the function 
\begin{equation}\label{log-linear-part-1}
\B{\phi} \mapsto h(\B\phi) := \sum_{i=1}^{r} \left ( \log\left(\R{max}\{1,\lambda_i^{*}\}\right) -\R{max}\{1,\lambda_i^{*}\} +1 \right) 
\end{equation}
 is concave on $\B\phi \geq \M{0}$; where, $\{\lambda_i^{*}\}_{1}^{p}$ are the eigenvalues of $\M{S}^*$.
%  $\M{S^{*}} = \M{\Phi^{\tfrac{1}{2}} S \Phi^{\tfrac{1}{2}} }$ and $\lambda_1^{*} \geq \lambda_2^{*} \geq \ldots \geq \lambda_p^{*}$ are the eigenvalues of $\M{S^{*}}$.
\begin{proof}
%%From Corollary~\ref{corr-eigS*Phi=eigPhi*S} we have 
Note $\B\lambda\left(\M{S}^{*}\right)=\B\lambda\left(\M{S}^{\tfrac{1}{2}} \B\Phi \M{S}^{\frac12}\right)$ (Corollary~\ref{corr-eigS*Phi=eigPhi*S}).
%$h(\B\phi)$ is an  eigenfunction, i.e., it 
%%%\footnote{that is, the function depends only upon the eigenvalues of $\M{S} ^{\tfrac{1}{2}} \B\Phi \M{S} ^{\tfrac{1}{2}}$}
%depends upon the eigenvalues of  $\M{S} ^{\tfrac{1}{2}} \B\Phi \M{S} ^{\tfrac{1}{2}}$:
By a classic result due to Davis~\citep{davis-57,Lewis96convexanalysis} the following mapping
\begin{equation} \label{log-lin-part2}
\M{S} ^{\tfrac{1}{2}} \B\Phi \M{S}^{\tfrac{1}{2}} \mapsto \sum_{i=1}^{r} \left( \log\left(\R{max}\{1,\lambda_i^{*}\}\right) -\R{max}\{1,\lambda_i^{*}\} +1 \right)
\end{equation}
 is concave in $\M{S} ^{\tfrac{1}{2}} \B\Phi \M{S}^{\tfrac{1}{2}}$ if and only if the 
function~\eqref{top-sum11} is 
symmetric\footnote{A function $g(y_{1}, \ldots, y_{p}) : \Re^{p} \rightarrow \Re$ is said to be symmetric in its arguments if,  for any permutation $\pi$ of the indices $\{1, \ldots, p \}$, we have 
$g(y_{1}, \ldots, y_{p}) = g(y_{\pi(1)}, \ldots, y_{\pi(p)})$} and concave in $\M{y}$ on $\M{y} \geq \M{0}$. 
It is easy to see that the function in~\eqref{top-sum11} is symmetric and concavity follows from Proposition~\ref{lem-diff-convex1-0}.
So we conclude that the map in~\eqref{log-lin-part2} is concave in $\M{S} ^{\tfrac{1}{2}} \B\Phi \M{S}^{\tfrac{1}{2}}$. 
The linearity of the map $\B\phi \mapsto \M{S} ^{\tfrac{1}{2}} \B\Phi \M{S}^{\tfrac{1}{2}}$ implies that 
%%~\eqref{log-lin-part2} is concave in $\B\Phi$~\cite{BV2004}, and hence 
$h(\B\phi)$ is concave in $\B\phi$ on $\B\phi \geq \M{0}$. This
completes the proof of the proposition.
\end{proof}
\end{prop}
%%\begin{cor}\label{corr:concave-2nd-part}
%%For any  $r \in \{1, \ldots, p\}$ 
%%the function 
%%\begin{equation}\label{log-linear-part-1}
%%\B{\phi}:= \left(\phi_{1}, \ldots, \phi_{p}\right) \mapsto -h\left(  \B\phi \right) := - \left(\sum_{i=1}^{r} \left ( \log\left(\R{max}\{1,\lambda_i^{*}\}\right) -\R{max}\{1,\lambda_i^{*}\} +1 \right) \right)
%%\end{equation}
%% is convex in $\B{\phi}$, where, $\phi_{i}  \geq 0$ for $i = 1, \ldots, p$.
%%\end{cor}

%%%%The above minimization problem can be solve easily by differentiating~\eqref{eqn-blkdiag-relaxed-fn} w.r.t. $\M{\Phi}_i$, and the minimizer is give by
%%%%\begin{equation}
%%%%(\M{\Phi}_i)^{(k+1)} = \left(\M{S}_{ii} - \M{\nabla}_{k,ii} \right)^{-1}
%%%%\end{equation}
%%%%Where 
%%%%$$
%%%%\M{\Delta}_k= \dfrac{\partial f_2({\B{\Phi}})}{\partial \B\Phi}\bigg|_{\B\Phi = \B\Phi ^{(k)}} =\left[\begin{array}{cccc}
%%%%\M{\Delta}_{k,11}&\M{\Delta}_{k,12}&\cdots &\M{\Delta}_{k,1m}\\
%%%%\M{\Delta}_{k,21}&\M{\Delta}_{k,22}&\cdots &\M{\Delta}_{k,2m}\\
%%%%\vdots & &\ddots &\vdots \\
%%%%\M{\Delta}_{k,m1}&\M{\Delta}_{k,m2}&\cdots &\M{\Delta}_{k,mm}\\
%%%%\end{array}\right]
%%%%$$
%%%%$\M{\Delta}_{k,ij}$ is the $p_i \times p_j $ submatrix of $\M{\Delta}_k$. $\M{\Phi}^{(k)}$ is the value of $\M{\Phi}$ at $k^{\R{th}}$ iteration.

%% in the presence of additional convex constraints
%%%is also convex. For example, if $\B\Phi$ is block-diagonal; or $ \B\Phi $

\begin{cor}\label{corr:diff-cvx1}
For any  $\B{\phi} > \M{0}$, $f(\B\phi)$
can be written as the difference of two convex  functions, $f_{i}(\B\phi), i = 1, 2$ that is:
$f(\B{\phi}) = f_1(\B{\phi}) - f_2(\B{\phi})$, where, 
\begin{equation}\nonumber
%%\begin{aligned}
%f(\B{\phi}) =& f_1(\B{\phi}) - f_2(\B{\phi}), ~~~\text{where,} \\
f_1(\B{\phi}) = \sum_{i=1}^{p} \left( -\log \phi_i +   s_{ii}\phi_i \right)~~~\text{and}~~~
f_{2}(\B{\phi}) = - \sum_{i=1}^{r} \left( \log(\R{max}\{1,\lambda_i^{*}\}) -\R{max}\{1,\lambda_i^{*}\} +1 \right).
%%\end{aligned}
\end{equation}
%%%and $\B\phi \mapsto f_i(\B\phi), i = 1, 2$ are  convex on $\B\phi > \M{0}$.
%In particular, for $q=1$, the function~\eqref{diff-cvx-prog1} is  concave due to the linearity of $\M{A} \mapsto \|\M{A}\|_1 = \tr(\M{A})$.
%where $\M{\Phi}=\R{diag}=(\B{\phi})$ and $\lambda_1^{*} \geq \lambda_2^{*} \geq \ldots \geq \lambda_p^{*}$ are the eigenvalues of $\M{\Phi^{-\tfrac{1}{2}}  S  \Phi^{-\tfrac{1}{2}} }$
\begin{proof}
The convexity of $f_{1}(\B\phi)$ is easy to see. Proposition~\ref{lem-diff-convex1} implies that $f_{2}(\B\phi)$ is convex. 
\end{proof}
\end{cor}
%%%%\subsection{Some special cases of Problem~\eqref{obj-2}} \label{sec-cvx-lowerbound}
%%%When the sample covariance matrix is of full rank we can deduce some nice properties of Problem~\eqref{obj-2}.
%%%\subsubsection*{Spectral function:}
%%%\textcolor{red}{RM: Is prop - 6 useful in the computations or do we work with the DC form of Cor 3 always?}

When $\M{S}$ is full rank, i.e., $\M{S} \succ \M{0}$, then Problem~\eqref{obj-2} can be rewritten purely as a function of the eigenvalues $\{\lambda_{i}^*\}_{i \geq 1}$ -- this is established in Proposition~\ref{Spectral function}.
%%%{\Large{\textcolor{red}{RM: Avoid writing $\lambda^*_{i}$s --- write $\{\lambda_{i}^*\}_{i \geq 1}$ instead.}}}
%%Such a representation does not seem to be available when $\M{S}$ is rank deficient.
%%% it is not clear how to express Problem~\eqref{obj-2} as a function of eigenvalues. 
%%Problem~\eqref{obj-2} is not a spectral function since $f_{1}(\B\phi)$ is not, whereas $f_{2}(\B\phi)$ can be written purely in terms of eigenvalues of the matrix $\M{S}^*$. 

\begin{prop}\label{Spectral function}
If $\M{S}\succ \M{0}$ Problem~\eqref{obj-2} is equivalent to the following problem:
\begin{equation}\label{obj-full}
\begin{array}{lll}
 \mini\;\; & 
\displaystyle  \bar{f}(\B\phi):= \sum_{i=1}^{p} \left(- \log \lambda_i^{*} +   \lambda_i^{*} \right) + \sum^{r}_{i=1} \left( \log(\R{max}\{1,\lambda_i^{*}\}) -\R{max}\{1,\lambda_i^{*}\} +1 \right)   \\
\sbt \;\;\; & \M{0} \prec  \M{\Phi}=\mathrm{diag}(\phi_1, \ldots, \phi_p) \preceq \frac{1}{\epsilon} \M{I},
\end{array}
\end{equation}
where, $\{\lambda_{i}^*\}_1^{p}$ are the eigenvalues of $\B{\Phi}^{\tfrac{1}{2}} \M S \B{\Phi}^{\frac12}$.
% where, $\M{S^{*}} = \M{\Phi^{\tfrac{1}{2}} S \Phi^{\tfrac{1}{2}} }$ and $\lambda_1^{*} \geq \lambda_2^{*} \geq \ldots \geq \lambda_p^{*}$ are the eigenvalues of $\M{S^{*}}$.
\begin{proof}
Problem~\eqref{obj-2} is equivalent to 
minimizing $\bar{f}(\B\phi):=-\log\det(\M{S}) + f(\B{\phi})$ over $\M{0} \prec  \M{\Phi} \preceq \frac{1}{\epsilon} \M{I}.$ 
The function $\bar{f}(\B\phi)$ can be expressed as:
 %%Denoting this new objective function as $\bar{f}(\B\phi)$, we have:
%the objective function: $\bar{f}(\B\phi):= - \log\det(\M{S}) + f(\B\phi)$ over $ \M{0} \prec \B\Phi \preceq \frac{1}{\epsilon} \M{I}$.
%\begin{equation*} 
%\begin{array}{lll}
% \;\;\;\;\;\;\; & \mini\;\; & 
%\displaystyle \bar{f}(\B\phi):= - \log|\M{S}| +\sum_{i=1}^{p} [- \log \phi_i +   \M{S}_{ii}\phi_i ] + \sum^{r}_{i=1} \left[ \log(\R{max}\{1,\lambda_i^{*}\}) -\R{max}\{1,\lambda_i^{*}\} +1 \right] \\
%&\sbt \;\;\; & \M{\Phi}=\mathrm{diag}(\phi_1, \ldots, \phi_p) \succ \M{0},
%\end{array}
%\end{equation*}
%Now we can simplify $\bar{f}(\B\phi)$ as
\begin{align}
\bar{f}(\B\phi)=&\sum_{i=1}^{p} \left(-\log \phi_i +   s_{ii}\phi_i \right) - \log\det(\M{S})  +  \sum^{r}_{i=1} \left( \log\left(\max\{1,\lambda_i^{*}\}\right) -\max\{1,\lambda_i^{*}\} +1 \right) \label{line-0-1}\\
%=&- \log\det(\M{ \Phi}^{\tfrac{1}{2}} \M{S} \B{\Phi} ^{\tfrac{1}{2}} )  +   \R{tr}(\B{ \Phi} ^{\tfrac{1}{2}} \M{S} \B{\Phi} ^{\tfrac{1}{2}} ) + \sum^{r}_{i=1} \left( \log\left(\max\{1,\lambda_i^{*}\}\right) -\max\{1,\lambda_i^{*}\} +1 \right) \ \label{line-0-2} \\
=&- \log\det(\M{S}^*)  +   \R{tr}(\M{S}^*) + \sum^{r}_{i=1} \left( \log\left(\max\{1,\lambda_i^{*}\}\right) -\max\{1,\lambda_i^{*}\} +1 \right) \ \label{line-0-2} \\
=& \sum_{i=1}^{p} \left( - \log \lambda_i^{*} +   \lambda_i^{*} \right) + \sum^{r}_{i=1} \left( \log\left(\R{max}\{1,\lambda_i^{*}\}\right) -\R{max}\{1,\lambda_i^{*}\} +1 \right)  \label{line-0-3}
\end{align}
where, line~\eqref{line-0-2} follows from~\eqref{line-0-1} by observing that
\begin{equation}\label{equality-11}
\begin{aligned}
\sum_{i=1}^{p}( -\log(\phi_{i}) + s_{ii} \phi_i) -\log\det(\M{S}) =& -\log\det(\B\Phi) + \tr(\M{S}\B\Phi) -\log\det(\M{S}) \\
=& -\log\det(\M{S}\B\Phi) + \tr(\M{S}\B\Phi) \\
=& -\log\det(\M{S}^*) + \tr(\M{S}^*),
\end{aligned}
\end{equation}
where, $\M{S}^*=\M{ \Phi}^{\tfrac{1}{2}} \M{S} \B{\Phi} ^{\tfrac{1}{2}}$; and the last equality in~\eqref{equality-11} made use of Corollary~\ref{corr-eigS*Phi=eigPhi*S}.
Moreover, note that since $\M{S}$ is of full rank and $\B\Phi \succ \M{0}$; all the eigenvalues $\lambda_i ^{*}> 0$. 
%%So $\log\left(\lambda_i ^{*}\right)$ is always well defined.
This completes the proof.
\end{proof}
\end{prop}

Proposition~\ref{Spectral function} provides an interesting alternate characterization of the
formulation presented in Corollary~\ref{corr:diff-cvx1} -- this helps us gain additional understanding of the optimization problem for ML factor analysis when $\M{S} \succ \M{0}$. 

\subsubsection{Computing subgradients} \label{sec:compute-subgrad-1}
Herein, we study the computation of (sub)gradients~\cite{rockafellar2009variational} of properties of the functions $f_{1}(\B\phi)$ and $f_{2}(\B\phi)$.
Note that $f_{1}(\B\phi)$ is differentiable. However, the convex (spectral) function $f_{2}(\cdot)$ is not differentiable. 
%%% but its subgradient can be computed using results from subdifferentials 
%%%and in fact its gradient is Lipschitz on $[\tfrac{1}{\epsilon}, 0)^{p}$ with Lipschitz constant $$.
A subgradient  of $f_{2}(\B\phi)$ can be computed following the work of~\cite{lewis-96} on differentiability of spectral functions of 
Hermitian matrices. 
%Let us consider the following function:
%\begin{equation}
%(y_{1}, \ldots, y_{p}) \mapsto \widetilde{f}_{2} (\M{y}) := - \sum_{i=1}^{r} \left( \log\left(\R{max}\{1,y_{(i)}\}\right) -\R{max}\{1,y_{(i)}\} +1 \right), 
%\end{equation}
%on $\M{y}  \geq \M{0}$, where, $y_{(i)} \geq y_{(i+1)}$ for all $i \geq 1$.
%\begin{align*}
%f_2(\B{\phi})& = - \sum_{i=1}^{r} \left( \log\left(\R{max}\{1,\lambda_i^{*}\}\right) -\R{max}\{1,\lambda_i^{*}\} +1 \right)
%\end{align*}
%where $h(\lambda_i^{*})= \log\left(\R{max}\{1,\lambda_i^{*}\}\right) -\R{max}\{1,\lambda_i^{*}\} +1 $ and  $\lambda_1^{*} \geq \ldots \geq \lambda_p^{*}$ are the eigenvalues of $\M{ \Phi^{\tfrac{1}{2}} S \Phi^{\tfrac{1}{2}} }$.
To this end, consider the representation of $H_{r}(\M{y})$ in~\eqref{point-wise-inf1};
and define the function 
$g(y)= \log\left(\R{max}\{1, y \right\}) -\R{max}\{1, y \} +1$ on $y \geq 0$. Define $\widetilde{H}_{r}(\M{y}) = - H_{r}(\M{y})$ which is a convex function in $\M{y}$.
If $\partial \widetilde{H}_{r}(\M{y})$ is a subgradient of $\widetilde{H}_{r}(\M{y})$, then it can be computed by using  Danskin's theorem:
%\begin{equation}\label{subgrad-1}
$\partial \widetilde{H}_{r}(\M{y}) = -\sum_{i=1}^{p} \hat{w}_{i} \nabla g(y_{i}),$
%\end{equation}
where, $\hat{\M{w}}$ is a minimizer of the inner optimization task in Problem~\eqref{point-wise-inf1}; and $\nabla g(y_{i}) \in \Re^{p}$ is the gradient of $g(y_{i})$, with $i$th coordinate
given by $\nabla_{i} g(y_{i}) = \min \{ 0,  \frac{1}{y_{i}} - 1 \},$ and $\nabla_{j} g(y_{i})=0$ for all $j \neq i$.
The function $\widetilde{H}_{r}(\M{y})$ is  differentiable at $\M{y}$ iff $\hat{\M{w}}$ is unique\footnote{We note that this occurs if $g(y_{(r+1)})=g(y_{(r)})$.}. The set of all subgradients of $\widetilde{H}_{r}(\M{y})$ is
given by 
$$\text{Conv} \left( \left \{ -\sum_{i=1}^{p} \hat{w}_{i} \nabla g(y_{i}):  \text{$\hat{\M{w}}$ is a minimizer of Problem~\eqref{point-wise-inf1}} \right \} \right). $$
%$\text{Conv}( \left \{ \sum_{i=1}^{p} \hat{w}_{i} \nabla g(y_{i}) \right \}),$ where, $\hat{w}$ ranges over all solutions to Problem~\eqref{point-wise-inf1}.

Let us consider a matrix $\M{A} \succeq \M{0}$, with eigen decomposition  $\M{A} = \M{V}  \diag(\B\lambda)\M{V}^\top$, and consequently consider the spectral convex function
$ \tilde{g}_{r}(\M{A}) := -\sum_{i=1}^{r} g(\lambda_{i}).$ 
%The function $\lambda \mapsto g(\lambda)$ on $\lambda \geq 0$  is differentiable with $\nabla g(\lambda) = \min \{ 0, 1 - \frac{1}{\lambda} \}.$
Using properties of subgradients of spectral functions~\cite{lewis-96}, we have that a subgradient of $\M{A} \mapsto \tilde{g}_{r}(\M{A})$ is given by:
\begin{equation}
\partial \tilde{g}_{r}(\M{A}) = \M{V}  \diag( \partial \widetilde{H}_{r}(\B\lambda) ) \M{V}^\top,
\end{equation}
where, $ \partial \widetilde{H}_{r}(\B\lambda) $ is a subgradient of $\B\lambda \mapsto \widetilde{H}_{r}(\B\lambda)$.
%%where, $\nabla g(\B\lambda(A))$ is a vector, where, its $i$the entry is given by:
%%\[ \delta_i = \left\{ 
%%  \begin{array}{l l}
%%    \R{ min} \left(0,1-\dfrac{1}{\lambda_i^{*}}\right) & \quad \text{if $ 1 \leq i \leq r $ }\\
%%    0 & \quad \text{o.w. }
%%  \end{array} \right.\]
%%
%%
%Note that this spectral function admits the following representation:
%\begin{equation}
%\begin{aligned}
%\sum_{i=1}^{r} g(\lambda_{i}) =&\max_{\M{W}}& \; \langle g(\M{A}) ,  \M{W} \rangle  \\
%& \sbt& \; \M{I} \geq \M{W} \geq \M{0} \\
%&& \; \tr(\M{W}) = r,
%\end{aligned}
%\end{equation}
%where, $g(\M{A}):= \M{U}_{A} \diag(g(\B\lambda)) \M{V}_{A}'$, with $g(\B\lambda)$ is a vector with $i$th coordinate given by: $g(\lambda_{i})$.
%The function $\lambda \mapsto g(\lambda)$ is differentiable with $\nabla g(\lambda) = \min \{ 0, 1 - \frac{1}{\lambda} \}.$
%Recall that $f_2 \left( \B\phi \right)$ is an eigenfunction and we will calculate $\left(\dfrac{\partial f_2\left(\B{\phi}\right)}{\partial \phi_1},\ldots,\dfrac{\partial f_2\left(\B{\phi}\right)}{\partial \phi_p}\right)$ using ~\cite{lewis-96}. 

%%%{\Large{\bf{Why not just do the derivative w.r.t. $S^{\frac12}\B\Phi S^{\frac12}$? }}}

If $\B{ \Phi} ^{\tfrac{1}{2}} \M{S} \B{\Phi} ^{\tfrac{1}{2}}  = \M{ U}\R{diag}(\lambda_{1}^{*},...,\lambda_{p}^{*})\M{U }^{\top} $ is the eigen decomposition of 
$\M{ \Phi} ^{\tfrac{1}{2}} \M{S} \B{\Phi} ^{\tfrac{1}{2}}$; then using the  chain rule, the $i$th coordinate of $\partial f_{2}(\B\phi)$ is given by:
\begin{equation}
\partial_{i} f_2\left(\B{\phi}\right) =  % \dfrac{d \sqrt{\phi_k} }{d \phi_k} \times \dfrac{\partial f_2\left(\B{\phi}\right)}{\partial \sqrt{\phi_k}} \nonumber \\
 \left(\B{\Phi} ^{-\tfrac{1}{2}} \M{U}\M{D}_1\M{U} ^{\top} \B\Phi ^{\tfrac{1}{2}} \M{S}\right)_{ii}, \label{subgrad-exp}
\end{equation}
where, %$\M{D}= \R{diag}(\lambda_{1}^{*},...,\lambda_{p}^{*})$ with  
  $\M{D}_1 = \R{ diag }\left(\delta_1,...,\delta_p\right)$, with
\[ \delta_i = \left\{ 
  \begin{array}{l l}
    \R{max} \left \{ 0,1-\tfrac{1}{\lambda_i^{*}}\right \} & \quad \text{if $ 1 \leq i \leq r $ }\\
    0 & \quad \text{otherwise.}
  \end{array} \right.\]

%%%Finally,  note that $f(\B\phi)$ is not differentiable, but a sub-differential~\cite{borwein-conv} of $f(\B\phi)$ is given by, 
%%%$\partial f(\B\phi) = \nabla f_{1} (\B\phi) - \partial f_{2} (\B\phi),$ where, $\partial f_{2} (\B\phi)$ is a subgradient of $f_{2}(\cdot)$.

\subsection{Algorithm for Problem~\eqref{obj-2}: A DC optimization approach}\label{sec-algo}
%We describe the algorithm for Maximum Likelihood estimation of Factor Analysis parameters using Concave Convex Procedure.
% We present algorithms for minimizing 
 Problem~\eqref{obj-2} is a nonconvex optimization problem with semidefinite constraints and obtaining 
 a global minimum for a general $r$ is quite challenging. We thus focus on developing efficient computational procedures 
 for obtaining good (feasible) solutions to Problem~\eqref{obj-2}.  
By Corollary~\ref{corr:diff-cvx1}, Problem~\eqref{obj-2} is equivalent to the following nonconvex optimization problem: 
\begin{equation}\label{obj-1}
\begin{aligned}
 \mini&~~~~~~f(\B{\phi})=f_1(\B{\phi})-f_2(\B{\phi}) \\
 \sbt&~~~~~~~ \B\phi \in {\mathcal C}:=\{ \B\phi: \tfrac{1}{\epsilon} \geq  \phi_{i} > 0, i \in [p]\}.
\end{aligned}
\end{equation}
%%where
%%\begin{align} \label{2-cvx-functions}
%% f_1(\B{\phi})&= \sum_{i=1}^{p} \left( - \log \phi_i +   s_{ii}\phi_i \right) \nonumber \\
%% f_2(\B{\phi}) &= - \sum^{r}_{i=1} \left( \log(\R{max}\{1,\lambda_i^{*}\}) -\R{max}\{1,\lambda_i^{*}\} +1 \right)
%%\end{align}
%%$f_1(\B{\phi})$ and $f_2(\B{\phi})$ are convex in $\B{\phi}$.
%% $\lambda_1^{*} \geq \lambda_2^{*} \geq \ldots \lambda_p^{*}$  are eigenvalues of $\M{\Phi^{\tfrac{1}{2}} S \Phi^{\tfrac{1}{2}}}$.\\
%We will solve Problem~\eqref{obj-1} by Concave-Convex Procedure (CCCP)~\cite{cccp_03}. 
%%%To obtain good feasible solutions to Problem~\eqref{obj-1} 
We use a sequential linearization procedure: wherein, at every iteration, 
we linearize the function $f_{2}(\B\phi)$ (leaving $f_{1}(\cdot)$ as is) and solve the resultant convex problem. This is an instance of 
the well-known difference of convex optimization based algorithms~\cite{hiriart1985generalized,tuy1995dc,pham2013convergence} or DC algorithms in short. 
In the machine learning community, these methods are also known as the convex concave procedure~\cite{gert-09,cccp_03}.
These algorithms have gained significant traction in the wider optimization community, especially recently, due to their pervasive use in practice -- some excellent recent works 
on this topic include~\cite{ahn2017difference,dinh2014recent,nouiehed2017pervasiveness,pang2016computing} (see also references therein).
However, to our knowledge, ours is the first paper to use this approach in the context of ML factor analysis. We will also like to remind the reader that 
the reformulations in Section~\ref{sec:reform} play a key role in arriving 
at decomposition~\eqref{obj-1} --- this sets the stage for the application of the DC approach.

%This is a sequential linear optimization algorithm which is used to optimize functions that can be represented as a difference of two convex functions. 
%This method has certain advantages over standard gradient descent techniques  since they do not require estimating a
%step size and are often found empirically to converge rapidly. 
%\cite{gert-09} discuss about the asymptotic convergence guarantees of this procedure by providing an MM interpretation to these methods. 
% %the convergence of CCCP.  
%Herein, we present an elementary analysis of this sequential SDP procedure; and also present finite time convergence rate guarantees to a first order stationary point. 
% We initialize the algorithm with an initial feasible solution $\B{\phi}^{(1)}$.
%\footnote{The constraint $\M{\Phi}=\R{diag}(\phi_1,\phi_2,\ldots,\phi_p) \succ \M{0}$ is equivalent to $\phi_i > 0 \;\; \forall 1 \leq i \leq p $} $\B{\phi}^{(1)}$. 

Let us formally describe the algorithm. If $\B{\phi}^{(k)} \in {\mathcal C}$ denotes the value of $\B\phi$ at the $k${\text{th}} iteration, we linearize $f_2(\B{\phi})$ at $\B{\phi}^{(k)}$ with $f_1(\B{\phi})$ unchanged and obtain a convex approximation of $f(\B{\phi})$, denoted by $F(\B\phi; \B\phi^{(k)})$; and this is given by:
\begin{equation}\label{convex-mazoriztion}
f(\B{\phi}) \approx f_1(\B{\phi})- \left( f_2\left(\B{\phi}^{(k)}\right) + \left\langle \B\nabla_{k},\B{\phi}-\B{\phi}^{(k)} \right\rangle \right) := F(\B\phi; \B\phi^{(k)}),
\end{equation}
where, $\B\nabla_{k}$ is a subgradient of $f_2 (\B{\phi} )$ at $\B{\phi}^{(k)}$ (Section~\ref{sec:compute-subgrad-1}). 
%A subgradient of $f_2 (\B{\phi} )$ is known to exist from standard results in convex analysis~\cite{borwein-conv}. The computation of
%$\nabla_{k}$ is described
%%%via a rank $r$ eigen decomposition, and is described 
%in Section~\ref{sec:compute-subgrad-1}.
%%%A subgradient $f_2 (\B{\phi} )$ at $\B{\phi}^{(k)}$, denoted by $\nabla_{k}:=\partial f_2 (\B{\phi}^{(k)})$ is given by :
%%%\begin{equation}
%%%\nabla_{k,i} =  % \dfrac{d \sqrt{\phi_k} }{d \phi_k} \times \dfrac{\partial f_2\left(\B{\phi}\right)}{\partial \sqrt{\phi_k}} \nonumber \\
%%% \left((\B\Phi^{(k)}) ^{-\tfrac{1}{2}} \M{U}\M{D}_1\M{U} ^{\top} ({\B\Phi}^{(k)})^{\tfrac{1}{2}} \M{S}\right)_{ii}, \label{subgrad-exp}
%%%\end{equation}
%%%where, $ \B\Phi^{(k)} = \R{diag}(\B\phi^{(k)}) $,  $(\B\Phi^{k}) ^{\tfrac{1}{2}} \M{S} (\B\Phi^{k}) ^{\tfrac{1}{2}}  = \M{ U}\R{diag}(\lambda_{1}^{*},...,\lambda_{p}^{*})\M{U }^{\top} $ is the eigen decomposition of 
%%%$(\B\Phi^{k}) ^{\tfrac{1}{2}} \M{S} (\B\Phi^{k}) ^{\tfrac{1}{2}}$  
%%% %$\M{D}= \R{diag}(\lambda_{1}^{*},...,\lambda_{p}^{*})$ with  
%%% and $\M{D}_1 = \R{ diag }\left(\delta_1,...,\delta_p\right)$, with
%%%\[ \delta_i = \left\{ 
%%%  \begin{array}{l l}
%%%    \R{ min} \left \{ 0,1-\dfrac{1}{\lambda_i^{*}}\right \} & \quad \text{if $ 1 \leq i \leq r $ }\\
%%%    0 & \quad \text{otherwise.}
%%%  \end{array} \right.\]
We compute $\B\phi^{(k+1)}$ as:
 %Simplifying the expression in~\eqref{convex-mazoriztion}, we arrive at the following optimization problem:
% for $f_1(\B{\phi})$ and $f_2(\B{\phi})$ in~\eqref{convex-mazoriztion} we obtain that minimizing~\eqref{convex-mazoriztion} with respect to $\B{\phi}$ is equivalent to solving the following minimization problem
%This is given by the following problem:
\begin{equation}\label{eqn:variational formulae for phi_k+1}
%\argmin_{ \frac{1}{\epsilon} \M{I} \succeq \M{\Phi} \succ \M{0}} \quad  \sum_{i=1}^{p} \left(- \log \phi_i +   s_{ii}\phi_i  - \left(\partial f_2\left({\B{\phi}}^{(k)}\right)\right)_i\phi_i\right)
\B\phi^{(k+1)} \in \argmin_{ \tfrac{1}{\epsilon}\M{1}  \geq \B{\phi} > \M{0}}~~F(\B\phi; \B\phi^{(k)}) =  \argmin_{ \tfrac{1}{\epsilon}\M{1}  \geq \B{\phi} > \M{0}}~~\sum_{i=1}^{p} \left(- \log \phi_i +   s_{ii}\phi_i  - \nabla_{k,i}\phi_i\right),
\end{equation}
where, $\nabla_{k,i}$ is the $i$th coordinate of $\B\nabla_{k} \in \Re^{p}$ (for all $i \in [p]$).
The $i$th entry of $\B\phi^{(k+1)}$ is given by:
$${\phi}^{(k+1)}_{i} = \min \{ \tfrac{1}{s_{ii} - \nabla_{k,i}} ,\tfrac{1}{\epsilon} \} ~~~ \text{for}~~~ i \in [p].$$
The updates continue till some stopping criterion is satisfied. This can be in terms of the relative change in the successive objective values:
%%\begin{equation}\label{stopping criterion}
$f(\B\phi ^{(k)}) - f(\B\phi ^{(k+1)}) <   \eta |f (\B\phi ^{(k+1)})|$
 or relative changes in successive iterate values: 
$\|\B{\phi}^{(k+1)} - \B{\phi}^{(k)}\|_{2} < \eta \| \B{\phi}^{(k)}\|_{2}$;
%%%\end{equation}
where, $\eta>0$ denotes a pre-specified tolerance level.  
We summarize the algorithm in the following display for convenience. 

%%%\textcolor{red}{RM: Shorten Algorithm~1; note subgradient computations are deferred to later in the paper}

%%%%%%%\begin{algorithm}[h] 
% {\fontsize{11pt}{11pt}\selectfont
 \begin{itemize}
\item[] \underline{{\bf{Algorithm~1}}: An algorithm for Problem~\eqref{obj-1}}.
\item[] Initialize with $\B{\phi}^{(1)}\in {\mathcal C}$ and update $\B\phi^{(k)} $ using~\eqref{eqn:variational formulae for phi_k+1} until 
some stopping criterion like $f(\B\phi ^{(k)}) - f(\B\phi ^{(k+1)}) <   \eta |f (\B\phi ^{(k+1)})|$ is met.

\end{itemize}

%\paragraph*{Convergence of Algorithm~\ref{algo:CCP}}:  Algorithm~\ref{algo:CCP} optimizes $f(\B\phi)$ using Concave-Convex-Procedure (CCCP). \cite{cccp_03} showed that CCCP always converges to a local optima or saddle point. The following proposition provides some deeper understanding of the convergence of Algorithm 1.

%~\cite{gert-09} gave an general framework for proving the convergence
%of CCCP. They also gave some deeper insight on the convergence properties when the convex functions in the the DC (difference of convex) representation are differentiable.     

\subsection{Computational guarantees for Algorithm~1}\label{sec:conv-ana}
We present herein, computational guarantees for Algorithm~1 in terms of the number of iterations required to deliver an approximate first order stationary point; and asymptotic convergence to a first order stationary point. 
Towards this end, we recall certain standard definitions of first order stationary conditions for Problem~\eqref{obj-1} (see for example,~\cite{pham2013convergence}). 
$\tilde{\B\phi} \in {\mathcal C}$ is said to be a first order stationary point of
Problem~\eqref{obj-1} if the following condition holds:
\begin{equation}\label{defn-fo-1}
\begin{aligned}
\tilde{\B\phi} \in& \argmin_{\B\phi} ~~ F(\B\phi; \tilde{\B\phi}) = f_{1}(\B\phi) - \langle \partial f_{2}(\tilde{\B\phi}), \B\phi -\tilde{\B\phi}  \rangle \\
                    &~~~~~~ \sbt ~~~~~ \B\phi \in {\mathcal C}= \{ \B\phi:\tfrac{1}{\epsilon} \M{1} \geq  \B\phi > \M{0}\},
\end{aligned}
\end{equation}
for some choice of a subgradient $\partial f_{2}(\tilde{\B\phi})$.
From standard optimality conditions of convex functions~\cite{rockafellar2009variational,borwein-conv}, the above condition is equivalent to saying that, there exists a subgradient $\partial f_{2}(\tilde{\B\phi})$ such that:
\begin{equation}\label{defn-fo-1-1}
 \partial f_{2}(\tilde{\B\phi}) \in  \nabla f_{1}(\tilde{\B\phi}) + {\mathcal N}(\tilde{\B\phi}; {\mathcal C}),  
 \end{equation}
where, ${\mathcal N}(\tilde{\B\phi}; {\mathcal C})$ is the normal cone to the convex set ${\mathcal C}$ at the point $\tilde{\B\phi}$.
%%% and $ \mathcal C $  is the constraint set on $ \B\phi $. 
Recall that ${\mathcal N}(\tilde{\B\phi}; {\mathcal C})$ is the convex cone of all vectors $\M{d} \in \Re^{p}$ such that
$ \langle \M{d}, \B\phi - \tilde{\B\phi}  \rangle  \geq 0$ for all $\B\phi \in {\mathcal C}$. In~\eqref{defn-fo-1-1} the 
right hand side denotes the standard Minkowski sum of a vector ($\nabla f_{1}(\tilde{\B\phi})$) and a set (${\mathcal N}(\tilde{\B\phi}; {\mathcal C})$).

Proposition~\ref{prop-diff phi goes to zero} shows that the sequence $\B\phi^{(k)}$ leads to a  decreasing sequence of objective values, where the amount of decrease is lower bounded by
the squared norm of successive difference of the iterates $\{\B\phi^{(k)}\}$.
\begin{prop} \label{prop-diff phi goes to zero}
Let $\B\phi^{(k)}$ be a sequence generated via Algorithm~1. Then, there exists $\rho \geq {\epsilon^2}$ such that for every $k \geq 1$:
 \begin{equation} \label{diff phi goes to zero}
 f\left( \B\phi ^{(k)} \right) - f\left( \B\phi ^{(k+1)} \right) \geq \frac{\rho}{2} \| \B\phi ^{(k+1)} - \B\phi ^{(k)} \|^{2}.
 \end{equation} 
%%%%for $k \geq 1.$

\begin{proof}
From convexity of $f_{2}(\B\phi)$ we have that,
\begin{equation} \label{f_2 cvx prop}
%f_{2}\left(\B\phi^{(k+1)}\right)\geq f_{2}\left(\B\phi^{(k)}\right)+\left\langle \B\phi^{(k+1)}-\B\phi^{(k)},\partial f_{2}\left(\B\phi^{(k)}\right)\right\rangle
f_{2}\left(\B\phi^{(k+1)}\right)\geq f_{2}\left(\B\phi^{(k)}\right)+\left\langle \partial f_{2}\left(\B\phi^{(k)}\right), \B\phi^{(k+1)}-\B\phi^{(k)}\right\rangle,
\end{equation}
where, $\partial f_{2}\left(\B\phi\right)$ is a subgradient of $f_{2}(\B\phi)$. Note that the function
$f_{1}(\B\phi) = \sum_{i=1}^{p} f_{1}(\phi_{i})$ is separable across the coordinates and each $\phi_{i} \mapsto f_{1}(\phi_{i})$ 
is strongly convex on the interval $(0,\frac{1}{\epsilon}]$ with coefficient of strong convexity:
$\min_{\phi \in (0,\frac{1}{\epsilon}]}  \frac{1}{\phi^2} = \epsilon^2$. If we denote $\rho~(\geq \epsilon^2)$ as the coefficient of strong convexity for the function 
$f_{1}(\B\phi)$ on $(0,\tfrac{1}{\epsilon}]^{p}$ then:
\begin{equation} \label{f_1 cvx prop}
f_{1}\left(\B\phi^{(k)}\right)\geq f_{1}\left(\B\phi^{(k+1)}\right)+ \left\langle \B\phi^{(k)}-\B\phi^{(k+1)},\nabla f_{1}\left(\B\phi^{(k+1)}\right)\right\rangle + \frac{\rho}{2}\|\B\phi^{(k+1)}-\B\phi^{(k)}\|^{2},
\end{equation} 
where, $\nabla f_{1}\left(\B\phi^{(k+1)}\right)$ is the derivative of $f_{1}(\B\phi)$.
By standard optimality conditions~\cite{borwein-conv} of Problem~\eqref{eqn:variational formulae for phi_k+1} we have:
\begin{equation}\label{geq-deriv-1}
\min \left\{ \left\langle \nabla F(\B\phi^{(k+1)}; \B\phi^{(k)}), \B\phi - \B\phi^{(k+1)}  \right \rangle :  \tfrac{1}{\epsilon} \M{1} \geq \B\phi > \M{0} \right\}  \;\;  \geq 0,
\end{equation}
where, $\nabla F(\B\phi^{(k+1)}; \B\phi^{(k)})=\nabla f_1(\B\phi^{(k+1)}) - \partial f_{2}( \B\phi^{(k)})$ is the derivative of $\B\phi \mapsto F(\B\phi; \B\phi^{(k)})$ evaluated at 
$\B\phi^{(k+1)}$. Adding~\eqref{f_1 cvx prop} and~\eqref{f_2 cvx prop}, and rearranging terms we get:
\begin{align*}
f_1\left( \B\phi^{(k+1)} \right)-f_2 \left( \B\phi^{(k+1)}\right) 
& \leq f_1\left( \B\phi^{(k)} \right)-f_2 \left( \B\phi^{(k)}\right) -
 \frac{\rho}{2}\|\B\phi^{(k+1)}-\B\phi^{(k)}\|^{2} \\
 & \phantom{=} \;\;\;\;\;\;\;\;\;\;\;\;\;\;\;\;\;\;\;\;\;\;\;\;+ \underbrace{\left\langle\B\phi^{(k+1)}-\B\phi^{(k)},\nabla f_{1}\left(\B\phi^{(k+1)} \right) -
 \partial f_{2}\left(\B\phi^{(k)}\right) \right\rangle}_{\leq 0}\\
 & \leq f_1\left( \B\phi^{(k)} \right)-f_2 \left( \B\phi^{(k)}\right) -
 \frac{\rho}{2}\|\B\phi^{(k+1)}-\B\phi^{(k)}\|^{2}.
 \end{align*}
Where, the last line follows from~\eqref{geq-deriv-1}
%%% observing that $ \partial F(\B\phi^{(k+1)}; \B\phi^{(k)}) = \partial f_{1} (\B\phi^{(k+1)}) - \partial f_{2} (\B\phi^{(k)} )$ and 
 by setting $\B\phi =  \B\phi^{(k)}$ in~\eqref{geq-deriv-1}.
%  and noting that $\langle \B\phi^{(k+1)}-\B\phi^{(k)},\nabla f_{1}(\B\phi^{(k+1)}) -
% \partial f_{2} (\B\phi^{(k)})\rangle \leq 0$. 
% Since we have that $f\left( \B\phi ^{(k)} \right) - f\left( \B\phi ^{(k+1)} \right) \rightarrow 0$ as $k \rightarrow +\infty $,  ~\eqref{diff phi goes to zero}  and $\rho > 0$ implies  $\| \B\phi ^{(k+1)} - \B\phi ^{(k)} \| \rightarrow 0 $.
 \end{proof}
 \end{prop}

%\begin{rem}\label{cor-proof-1}
%A careful inspection of the proof of Proposition~\ref{prop-diff phi goes to zero}, shows that~\eqref{f_1 cvx prop} can be strengthened to:
%\begin{equation} \label{f_1 cvx prop-mod}
%f_{1}\left(\B\phi^{(k)}\right)\geq f_{1}\left(\B\phi^{(k+1)}\right)+ \left\langle \B\phi^{(k)}-\B\phi^{(k+1)},\partial f_{1}\left(\B\phi^{(k+1)}\right)\right\rangle 
%+ \sum_{i=1}^{p} \frac{\rho_{i,k}}{2} (\phi^{(k+1)}_{i}- \phi_{i}^{(k)})^{2},
%\end{equation} 
%where, $\rho_{i,k} = \min \{ 1/{\phi^2_{i}}^{(k)},  1/{\phi^2_{i}}^{(k+1)} \}$; with $\phi^{(k)}_{i}$ denoting the $i$th coordinate of $\B\phi^{(k)}$.
%This leads to a generalization of~\eqref{diff phi goes to zero}:
% \begin{equation} \label{diff phi goes to zero-1}
% f\left( \B\phi ^{(k)} \right) - f\left( \B\phi ^{(k+1)} \right) \geq  \sum_{i=1}^{p} \frac{\rho_{i,k}}{2} (\phi^{(k+1)}_{i}-\B\phi_{i}^{(k)})^{2}.
% \end{equation} 
%\end{rem}

 The above proposition says that $f(\B\phi^{(k)})$ is a decreasing sequence; and being bounded below, converges to $\hat{f}$, say. 
%Note that $\hat{\B\phi}$ is a first order stationary point, if: 
%\begin{equation}\label{defn-fo-1}
%\hat{\B\phi} \in \argmin_{\B\phi}~ F(\B\phi; \hat{\B\phi})~~\sbt~~ \frac{1}{\epsilon} \M{1} \geq  \B\phi > \M{0}.
%\end{equation}
By the definition of a first order stationary point~\eqref{defn-fo-1},
the quantity $\|\B\phi^{(k+1)} - \B\phi^{(k)}\|_{2}$ dictates the proximity of $\B\phi^{(k)}$ to a first order stationary point; and an approximate first order stationary point. 
We have the following proposition, formalizing the rate at which the sequence $\B\phi^{(k)}$ approaches a first order stationary point.

\begin{prop}\label{prop:diff of phi goes to zero}
The sequence $f(\B\phi^{(k)})$ is decreasing and converges to $\hat{f}$. The finite time convergence rate is given by:
\begin{equation}\label{rate-conv-1}
\min_{1 \leq k \leq {\mathcal K} }  \rho\| \B\phi^{(k+1)} - \B\phi^{(k)} \|_{2}^2 \leq  \frac{2}{{\mathcal K}} \left(f(\B\phi^{(1)}) -\hat{ f} \right).
\end{equation}
\begin{proof} The proof uses~\eqref{diff phi goes to zero}. If $\Delta_{k}:= \frac{\rho}{2} \| \B\phi ^{(k+1)} - \B\phi ^{(k)} \|^{2}$ then:
%\begin{equation}\label{proof-rate-1}
%\begin{aligned}
%{\mathcal K} \min_{1\leq k \leq {\mathcal K} }   \frac{\rho}{2} \| \B\phi ^{(k+1)} - \B\phi ^{(k)} \|^{2} \leq& \frac{\rho}{2} \sum_{i=1}^{\mathcal K} \| \B\phi ^{(k+1)} - \B\phi ^{(k)} \|^{2}  \\
% \leq&  \sum_{i=1}^{\mathcal K} \left \{ f\left( \B\phi ^{(k)} \right) - f\left( \B\phi ^{(k+1)} \right) \right \}  \\
% \leq &  f\left( \B\phi ^{(1)} \right) - \hat{f},
%  \end{aligned}
% \end{equation}
\begin{equation}\label{proof-rate-1}
\begin{aligned}
{\mathcal K} \min_{1\leq k \leq {\mathcal K} } \Delta_{k} \leq& \sum_{i=1}^{\mathcal K}\Delta_{k}  \leq 
  \sum_{i=1}^{\mathcal K} \left \{ f\left( \B\phi ^{(k)} \right) - f\left( \B\phi ^{(k+1)} \right) \right \}  
 \leq &  f\left( \B\phi ^{(1)} \right) - \hat{f},
  \end{aligned}
 \end{equation}
where, the second inequality uses~\eqref{diff phi goes to zero}; and 
the final inequality used the fact that $f(\B\phi^{(k)}) \downarrow \hat{f}$. The result~\eqref{rate-conv-1} follows from combining the left and right parts of the inequality~\eqref{proof-rate-1}.
\end{proof}
\end{prop}

The above proposition states that for any tolerance $\delta>0$, there is an integer ${\mathcal K}=O(\tfrac{1}{\delta})$ such that for some $k \in [{\mathcal K}]$, the following holds: 
$\| \B\phi ^{(k+1)} - \B\phi ^{(k)} \|^{2} \leq \delta$.
%%% where, $ {\mathcal K} \leq \frac{2}{\rho} \left(f(\B\phi^{(1)}) -\hat{ f} \right)

The following proposition shows that all limit points of the sequence $\B\phi^{(k)}$ are first order stationary points. 
\begin{prop}\label{prop:All limit points are stationary point }
	Any limit point of the sequence $\B\phi^{(k)}$ is a first order stationary point for Problem~\eqref{obj-1}.
\end{prop}
\begin{proof}
	The proof is deferred to Appendix~\ref{prop:proof of all limit points are stationary point}.
\end{proof}

\subsection{Computational Cost}
\label{sec:Computational details}
In this section, we discuss the computational cost of Algorithm~1 and techniques for computational scalability to large problems.

\smallskip

\noindent {\bf{When $n > p$:}} The main computational cost of Algorithm~1 stems from computing a subgradient of $f_{2}(\B\phi)$ which requires a low rank eigen decomposition of $\M{S}^*$.
When $p$ is small relative to $n$ it is convenient to form and work with $\M{S}^*$ (a $p \times p$ matrix).
 Creating $\M{S}$ from $\M{X}$ costs $O(np^2)$ (this operation can be done once offline).  Computing $\M{S}^*$ from $\M{S}$ costs $O(p^2)$.
A full eigendecomposition of the $p \times p$ matrix $\M{S}^*$
using dense matrix factorization methods~\cite{golub2012matrix}  costs $O(p^3)$ -- this approach can be used to handle problems with $p$ up to a few thousand ($p \approx 3,000$, for example).  Note that this cost is of the same order as obtaining the unrestricted MLE of $\B\Sigma^{-1}$ given by $\M{S}^{-1}$ (assuming $\M{S}$  is invertible). 

\smallskip

\noindent {\bf{When $p \gg n$:}} In several applications of interest, $n$ is much smaller compared to $p$---a situation that occurs commonly in the modern high dimensional regime (in microarray data analysis applications for example, $n$ is at most a few hundred and $p$ is in the order of tens of thousands). In such cases, obtaining a full eigen decomposition of $\M{S}^*$ will cost 
$O(n^2p)$, which is linear in $p$ if $n \ll p$. This follows by observing that 
$\M{S} = \frac1n\M{X}^\top \M{X}$ (where, $\M{X}$ is mean-centered)
%%%(where we already centered the features  of $\M{X}$).
and an eigen decomposition of $\M{S}^*$, i.e., $\B\Phi^{\frac12} \M{S} \B\Phi^{\frac12} = ( \tfrac{1}{\sqrt{n}}\M{X}\B\Phi^{\frac12})^\top (\tfrac{1}{\sqrt{n}}\M{X}\B\Phi^{\frac12})$ 
can be obtained via a SVD of the $n \times p$ matrix
$\tfrac{1}{\sqrt{n}}\M{X}\B\Phi^{\frac12}$ -- and this costs $O(n^2p)$. Note that this cost is significantly smaller than a direct eigendecomposition on the $p\times p$ matrix
$\M{S}^*$, which is prohibitively expensive for large values of $p$.

In addition, there are certain costs associated with matrix multiplications. Indeed, a careful book-keeping allows us to operate with matrices that are low rank -- we never need to create or form a dense $p \times p$ matrix -- this is beneficial from a memory standpoint.
First of all, note that the computation of $\M{X}\B\Phi^{\tfrac12}$ costs $O(np)$.
In addition, one needs to compute the diagonal entries 
of $T:=\B{\Phi}^{-\tfrac{1}{2}} \M{U}\M{D}_1\M{U} ^{\top} \B\Phi ^{\tfrac{1}{2}} \M{S}$
as in~\eqref{subgrad-exp}. Note that $T$ is a $p \times p$ matrix; however, its diagonal entries can be computed without explicitly creating the matrix $T$. 
This follows by observing that $T$ is the outer product of two low rank matrices: $T:=T_{1} T_{2}$,
where, $T_{1} = \B{\Phi}^{-\tfrac{1}{2}} \M{U}\M{D}_1$ and $T_{2} = \M{U} ^{\top} \B\Phi ^{\tfrac{1}{2}} \M{S}$ -- here, 
$T_{1} \in \Re^{p \times r}$ (recall that $r< n$); and $T_{2} \in \Re^{r \times p}$. Computing the diagonal entries of $T$ has a cost $O(pr)$.
Note that $T_{2}$ has to be computed carefully using transitivity, by observing that $T_{2}=T_{21} \M{X}$ with 
$T_{21}=\tfrac1n \M{U} ^{\top} (\B\Phi^{\tfrac{1}{2}}  \M{X}') \in \Re^{r \times n}$ --- thus we can avoid creating/storing a matrix of size $p \times p$.

\smallskip

\noindent {\bf{When both $p,n$ are large:}} When both $n,p$ are large, direct SVD factorization methods above become computationally expensive. 
%%We describe alternatives below.
Approximate schemes for large scale low rank SVD decompositions are called for.
Approximate rank $r$ eigen decompositions of $\M{S}^*$ can be computed using incremental eigendecomposition techniques as described in~\cite{brand2006fast}; or methods based on the
 Lanczos method~\cite{propack} or block power iterations~\cite{golub2012matrix}. All these methods will cost roughly $O(p^2 r)$ and can be significantly smaller (both in terms of memory and computational cost) when compared to a full eigendecomposition with cost $O(p^3)$ for $r \ll p$.

\subsection{Solutions to Problem~\eqref{obj-2-0-0} when $\epsilon \approx 0$}
\label{existance-of-ML}
%In factor model we assumed that the error variance $\psi_{i} >0  \text{ }\forall \text{ } i $. So we need to find a minimizer $\M{\hat{\Psi}}$ of Problem~\eqref{obj-2-0-0} such that  $\hat{\psi}_i > 0 \text{ } \forall \text{ } 1 \leq i \leq p $. But it may happen  that the optimization problem in~\eqref{obj-2-0-0}   have solution at the boundary i.e at optimal value of $\M{\Psi}$ $\hat{\psi_i} = 0$ for one more $i$. Although~\eqref{algo:CCP} ensures that at each iteration the estimates are greater than zero i.e. $\psi_{i}^{(k)} > 0 \text{ } \forall \text{ } 1 \leq i \leq p \text{ and } \forall \text{ } k \geq 1$, but it may happen that they converge to zero. One way to prevent this problem is to use regularization . We used two regularizations to get rid of this problem .The first one uses lower bound on error variance i.e. we assume $\psi_i \geq \epsilon $ for some small positive $\epsilon$ and the other uses Ridge Regularization. The two  regularizations are described below
%In factor model we assumed that the error variance $\B\Psi \succeq \epsilon  \M{I}$, for some $\epsilon > 0$. We showed in ~\eqref{eqn:neglog-likehdd-bdd}, Corollary~\ref{stationary-psi-upd} and~\ref{cor-obj-2-0} that under this assumption the ML solution of Factor Model always exists. 
The conventional version of the 
ML factor analysis optimization problem is given by:
\begin{equation}\label{obj-2-0-0-no-lbd}
\begin{aligned}
\mini~~& -\log\det(\B{\Sigma}^{-1}) + \R{tr}(\M{\Sigma}^{-1}\M{S}) \\
\sbt~~~~& \M{\Sigma} = \M{\Psi} + \M{L} \M{L}^\top,\B\Psi = \diag(\psi_{1}, \ldots, \psi_{p}) \succ \M{0},
\end{aligned}
\end{equation}
which may be interpreted as a limiting version of Problem~\eqref{obj-2-0-0} with $\epsilon \rightarrow 0+$. 
There are certain technical difficulties with Problem~\eqref{obj-2-0-0-no-lbd} as it may be unbounded below and hence a ML estimator need not exist.
This can lead to numerically unstable algorithms which are highly undesirable from a practical standpoint.   
\cite{robertson2007maximum} discuss necessary and sufficient condition for boundedness of Problem~\eqref{obj-2-0-0-no-lbd}. 
 Problem~\eqref{obj-2-0-0-no-lbd} is bounded below under the following conditions: (a) If $\M{S}$ is of full rank or (b) If $\M{S}$ is rank deficient, then 
 $r < s -1$ where, $s$ denotes the number of nonzero coordinates in the sparsest nonzero vector (i.e., a nonzero vector with the maximal number of zero coefficients) 
  in the null space of $\M{S}$. Computing the sparsest nonzero vector in a subspace, however, is a combinatorially difficult problem. 
  Note that if the sample covariance matrix corresponds to that of a continuous random variable, then Problem~\eqref{obj-2-0-0-no-lbd} is bounded below with probability one. 
 Even if Problem~\eqref{obj-2-0-0-no-lbd} is bounded below, the minimum may not be attained --- the infimum is attained when 
 some coordinates of $\B\Psi$ are set to zero. These are known as Heywood cases~\cite{robertson2007maximum} -- and they are well-known to create numerical difficulties from a computational viewpoint and also 
 misleading inference from a statistical viewpoint. 

If $\hat{\B\Psi}_{\epsilon}$ is a solution to Problem~\eqref{obj-2-0-0}; then the limiting value of $\hat{\B\Psi}_{\epsilon}$ as $\epsilon \downarrow 0+$ will give a solution to Problem~\eqref{obj-2-0-0-no-lbd}, provided it exists. 
%%In order to obtain an upper bound to Problem~\eqref{obj-2-0-0-no-lbd}, we use the following scheme.
This inspires a simple continuation scheme by considering a sequence of $\epsilon$-values converging to a 
small number $\epsilon \downarrow \epsilon'$ (with $\epsilon':=10^{-6}$, say); and using warm-starts for Algorithm~1.
%%%If some of the indices are stuck
%The following scheme may be useful.
%%We obtain an upper bound to Problem~\eqref{obj-2-0-0} for a small value of $\epsilon'=10^{-4}$, say, using Algorithm~1. 
%For every value of $\epsilon_{j}$, we obtain an estimate of $\hat{\B\Psi}_{\epsilon_{j}}$ via Algorithm~1, using $\hat{\B\Psi}_{\epsilon_{j-1}}$  as a warm-start to Algorithm~1. 
%At the final value of $\epsilon=\epsilon_{J}$, say, 
Suppose, there is a subset of indices $\M{M} \subset \{1, \ldots, p\}$ 
such that, $\hat{\psi_{i}} = \epsilon'$ for all $i \in \M{M}$, and $\hat{\psi_{i}} > \epsilon'$ for all $i \notin \M{M}$, then we obtain an upper bound to Problem~\eqref{obj-2-0-0-no-lbd} by fixing 
$\psi_{i} = \epsilon'$ for all $i \in \M{M}$ and optimizing over the remaining $\psi_{i}, i \notin \M{M}$ values. This can be performed by a simple modification to Algorithm~1, 
in the update~\eqref{eqn:variational formulae for phi_k+1}, where, we update \emph{only} the $\phi_{i}$ values (recall that $\phi_{i} =  \psi_{i}^{-1}$ for all $i$) corresponding to $i \notin \M{M}$ and set the remaining $\phi_{i}$ values to $1/\epsilon'$.

\subsection{Ridge Regularization}
%In previous problems we minimized the factor analysis problem~\eqref{obj-1} subject to the constraint $ \left\lbrace \B\phi \leq \frac{1}{\epsilon} \mathbf{1} \right\rbrace $, which makes sure that $ \hat{\B\phi} $ does not go to $ \infty $  (or equivalently the diagonal entries of $ \hat{\B\Psi} $ does not go to zero). On another way to achieve the same goal is to add square $ l_2 $ norm of $ \B\phi $ (ridge regression penalty) as a penalty function  in problem~\eqref{obj-1}.     
Instead of considering a direct lower bound on $\psi_{i}$'s as in Problem~\eqref{obj-2-0-0}, we can also consider a ridge regularized version of Problem~\eqref{obj-2-0-0-no-lbd} given by:
%%For ridge regularization we minimize 
\begin{equation}\label{obj-3}
\mini~~~f(\B{\phi}) + \sum^{p}_{i=1}\gamma \phi_i^{2} ~~~~\sbt~~~~~\B{\phi}=\mathrm{diag}(\phi_1, \ldots, \phi_p) \succ \M{0},
\end{equation}
for some $\gamma >0$. Algorithm~1 can be adapted to Problem~\eqref{obj-3}
 %%%%minimize~\eqref{obj-3}.  The reduced convex-minimization problem becomes 
by changing update~\eqref{eqn:variational formulae for phi_k+1} as:
\begin{equation}\label{ridge onedim_mini_problem}
\B\phi^{(k+1)} = \argmin_{\phi_{i} > 0, \forall i} ~~~ \sum_{i\in [p]} \left\{ -\log \phi_i +   s_{ii}\phi_i + \gamma \phi_i^{2}  - \nabla_{k,i}\phi_i \right \},
\end{equation}
where, $\nabla_{k,i}$'s can be computed as in Section~\ref{sec:compute-subgrad-1}. 
The $i$th coordinate of $\B\phi^{(k+1)}$ is given by:
\begin{equation}\label{update-ridge-11}
\phi^{(k+1)}_{i} =  \frac{1}{4\gamma}\left(\nabla_{k,i} -s_{ii} + \sqrt{( s_{ii}  - \nabla_{k,i})^{2} + 8\gamma}\right),~~~ 1 \leq i \leq p.
\end{equation}
Note that
$s_{ii}  - \nabla_{k,i}$ is the $i$th diagonal of 
%%$$ \left(\B{\Phi} ^{-\tfrac{1}{2}} \M{U}\M{D}_1\M{U} ^{\top} \B\Phi ^{\tfrac{1}{2}} \M{S}\right)$$
%%%Note that $s_{ii}  - \nabla_{k,i} $ depends on 
the matrix $\M{S} - \B\Phi^{-\tfrac{1}{2}} \M{UD}_1\M{U}^{\top} \B\Phi ^{\tfrac{1}{2}} \M{S}$, which can be rearranged as:
\begin{equation} \label{algoupdate-explained}
%%%\M{S} - \B\Phi^{-\tfrac{1}{2}} \M{U}\M{D}_1\M{U}^{\top} \B\Phi ^{\tfrac{1}{2}} \M{S} = 
\B\Phi ^{-\tfrac{1}{2}}
 \left( \B\Phi ^{\tfrac{1}{2}} \M{S}  \B\Phi ^{\tfrac{1}{2}} - \M{UD}_1\M{U}^{\top} \B\Phi ^{\tfrac{1}{2}} \M{S}  \B\Phi ^{\tfrac{1}{2}}   \right) \B\Phi ^{-\tfrac{1}{2} }= \B\Phi ^{-\tfrac{1}{2}}
 \left(\M{U}{\M{\widetilde{D}}}\M{U}^{\top} \right) \B\Phi ^{-\tfrac{1}{2} } 
\end{equation}
where $\widetilde{\M{D}}$ is a diagonal matrix with diagonal entries given by
\[ {\widetilde{d}}_{ii} = \left\{ 
  \begin{array}{l l}
    1 & \quad \text{if $ 1 \leq i \leq r $  \text{ and } $\lambda_i ^{*} \geq 1$}\\
    \lambda_i ^{*} & \quad \text{otherwise. }
  \end{array} \right.\]
  This implies that the matrix in~\eqref{algoupdate-explained} is positive semidefinite; and in particular, 
the diagonal entries are nonnegative: $s_{ii} -\nabla_{k,i} \geq 0$ for all $i \in [p]$.
%%%Using the fact that $  a^2 + b^2 \leq (a+b)^2  $ when $ a,b\geq 0 $  
%%%shows that
This implies from~\eqref{update-ridge-11} that:
${\psi}^{(k+1)}_i= \tfrac{1}{\phi^{(k+1)}_i} \geq \sqrt{2\gamma}$. 
Hence ridge regularization keeps the estimated $\psi_i$ values bounded away from zero.
A continuation scheme similar to that described in Section~\ref{existance-of-ML} can be used with $\gamma \downarrow \gamma'$ (with $\gamma' = 10^{-8}$, say) to get a good solution to Problem~\eqref{obj-2-0-0-no-lbd}.

\section{Computational Experiments} \label{sec:computational-result}

We present computational experiments on our proposed Algorithm~1 (aka \textsc{Factmle}).
We compare its performance versus other popular approaches for ML factor analysis on synthetic and real-data examples.
All computations were done in Matlab on a Mac desktop with  32 GB RAM.

\subsection{Comparison across different methods}~\label{subs:four_compare}
In this section we compare the performance of the method proposed herein: \textsc{Factmle} with leading algorithms for ML factor analysis:
\texttt{Fa}, \texttt{factoran} and \texttt{Emfact}.

\smallskip
 
%\noindent {\textbf{Data description:}}
%We generated (the true) $\M{L}^0 \in \Re^{p \times r_0}$ (with $r_0\ll p$) with entries drawn iid from $N(\mu, \sigma^2)$.
%Different values of $\mu $ and $\sigma^2$ were chosen for different experiments. 
%The (true) unique variances  $\M{\Psi}^0= \R{diag}(\psi^0_1,\psi^0_2,\ldots,\psi^0_p )$  were generated either from an exponential distribution with mean $\delta > 0 $ and location parameter $0.1$ or from an uniform distribution on $( 0.1 ,\theta )$ with 
%$\theta > 0.1 $. We generated $n$ samples from a multivariate Gaussian distribution with mean $\M{0}$ and covariance 
%$\M{\Sigma}^0=\M{\Psi}^0 + \M{L}^0{\M{L}^0}^\top$.

\noindent {\textbf{Synthetic data generation:}}
We generated (the true) $\M{L}^0 \in \Re^{p \times r_0}$ (with $r_0\ll p$) with entries drawn iid from $N(\mu, \sigma^2)$.
For examples with $n>p$, shown in Figure~\ref{fig:Simulated Example a}, we set $ \mu = 10 $  and $ \sigma^2 = 1$.  
The (true) unique variances  $\M{\Psi}^0= \R{diag}(\psi^0_1,\psi^0_2,\ldots,\psi^0_p )$  were generated independently from an exponential distribution with mean $ 10 $. 
Once  $\B{\Psi}^0 $ and $\M{L}^0$ were generated, we created 
$\M{\Sigma}^0=\M{\Psi}^0 + \M{L}^0{\M{L}^0}^\top$. We generated $\M{X}_{n \times p}$
 from a multivariate Gaussian distribution with mean $\M{0}$ and covariance $\B\Sigma$.
The results were averaged over ten different replications of $\M{X}$.

%%{\Large{\textcolor{red}{RM: fa.m link needs to be checked -- it is broken}}}

\smallskip

\noindent {\textbf{Competing methods:}}
We compared our proposed method: Algorithm~1 (\textsc{Factmle}) with the following popular ML methods for Factor Analysis:

\begin{itemize}
\item[1.] \texttt{Factoran:} This is the widely used, native implementation of ML factor analysis in {Matlab}; and this code is based on the seminal work of~\cite{joreskog1967some}.
\item[2.] \texttt{Emfact:}  This is the widely popular EM algorithm based technique; and the specific implementation is based on~\cite{bai2012statistical}.
\item[3.] \texttt{Fa:}  This is an EM algorithm based technique for factor analysis~\cite{bishop06:_patter_recog_machin_learn}\footnote{This function is available as a part of 
Matlab's PRML toolbox~\url{https://www.mathworks.com/matlabcentral/fileexchange/55883-probabilistic-pca-and-factor-analysis?focused=6047050&tab=function} }. 
\end{itemize}
Of the above three methods, \texttt{Factoran} and \texttt{Emfact} apply only when $n > p$ -- thus we restrict our attention to {\texttt{Fa}} as the only competing method to 
\textsc{Factmle} for $n <p$.
 We do not include the method of~\cite{JMLR:v18:15-613} in our comparisons, since it optimizes a different criterion (not the maximum likelihood objective). In terms of scalability considerations, the method of~\cite{JMLR:v18:15-613}
is less scalable (since it requires performing the eigen decomposition of an unstructured $p \times p$ matrix) than \textsc{Factmle} especially if $n$ is small. 
%The convergence threshold in our algorithm were set at $ \texttt{tol} = 10^{-3} $; for the other algorithms we used their default value of tolerance.  
%%We took the convergence threshold for our algorithm as TOL=$10^{-7}$.
The standard implementation of \texttt{Factoran} uses correlation matrices; hence we transformed the  results obtained from \texttt{Fcatoran} for comparison with other algorithms.

\smallskip

\noindent {\textbf{Performance measures:}}  All algorithms are compared in terms of the quality of solution obtained, i.e., the objective value -- the criteria they are all set to optimize. We consider the negative log-likelihood ${\mathcal L}(\hat{\B\Sigma}) := - \log\det({\hat{\M\Sigma}}^{-1}) + \R{tr}({\hat{\M\Sigma}}^{-1}\M{S}),$
where, ${\hat{\M\Sigma}} = \hat{\M{\Psi} }  + {\hat{\M L} \hat{{\M L}}^\top}$, upon verifying that the estimates are feasible.

In addition, we also consider the run-times of the different algorithms. 
This however, depends upon the different convergence criteria employed by the different algorithms (as we explain below); and the quality of solution obtained. 
We note that the quality of solution (in terms of objective value) can be different for different algorithms since the optimization problem is nonconvex. 
Thus run-time of an algorithm is interpreted in the context of the objective value attained in our experiments. 

 Finally, we note that  Algorithm~1 can readily incorporate warm-starts --- they may be useful if 
 one desires a sequence of solutions to Problem~\eqref{obj-2-0-0} for different values of $r = 1, 2, \ldots $. 
 Other algorithms like: \texttt{Fa}, \texttt{Factoran} and \texttt{Emfact} do not allow for warm-start specification; and we used their default initialization strategy.

\smallskip

\noindent {\bf{Comparison results:}} We first consider a synthetic dataset with $n>p$. Figure~\ref{fig:Simulated Example a} shows the performances of different methods for synthetic data generated as above, with  $p=200, n=2200, r^0=8$, and different choices of $r$. We compare the performance of the different algorithms in terms of their default convergence criteria. 

For \textsc{Factmle}, the tolerance level $\eta$ (for convergence based on objective value difference, as explained in Section~\ref{sec-algo}) was set to $10^{-4}$; we set the maximum number of iterations to be 1000.
 For all other we choose their default convergence criteria
with maximum number of iterations set to $1000.$ 
Figure~\ref{fig:Simulated Example a} suggests that the performance of \textsc{Factmle} measured in terms of the negative log-likelihood is significantly better than all the other algorithms -- thereby suggesting that it does well in the task it was set to accomplish. The performance of \texttt{Emfact}, in terms of negative log-likelihood, is comparable to \textsc{Factmle} when $r$ is small. However, \texttt{Emfact} often encounters numerical difficulties (especially when $r$ is large) and produces negative 
 estimates of ${\B\Phi}$ -- this violates the condition $ \B\Phi \succ \B{0} $ and is highly undesirable from a statistical standpoint.  
 Consequently, one should be cautious while using this method.
A highly attractive feature of \textsc{Factmle} is the timings. The number of iterations required in \texttt{Fa}, \texttt{Emfact} and $\texttt{Factoran}$ is much larger than that of \textsc{Factmle} (note that in addition, their objective values are worse as well). For example, for $p=200,r=6$ the (average) number of iterations for \texttt{Fa}, \texttt{Emfact},\texttt{Factoran} were around 1000 (i.e., the maximal iteration limit), but for \textsc{Factmle} it was less than twenty.
 We observe that the difference among {\textsc{Factmle}} and its competitors is more pronounced for larger values of $r$. 
 In terms of timings, {\textsc{Factmle}} is a clear winner --- and the differences are more pronounced with larger values of $p$ (as we see in subsequent experiments). 
 The performance \texttt{Fa} is better than \texttt{Factoran} and \texttt{Emfact}; and seems to be the only competitor to {\textsc{Factmle}}.
We perform a more detailed experiment comparing the performance of \texttt{Fa} with Algorithm~1 for two different types of initializations: with warm start and with random initialization in 
Figure~\ref{fig:Simulated Example a} (bottom panel).  We see that in both cases \textsc{Factmle} outperforms \texttt{Fa} in terms of quality of solution and also run-times. The timings of \textsc{Factmle} with warm start is  found to be slightly better than \textsc{Factmle} with cold start.  

 We also took some other values of $(n,p,r)$ but  the results were found to be quite similar; and hence we do not report them here.

\begin{figure}[h!]
\centering
\resizebox{0.95\textwidth}{0.3\textheight} {\begin{tabular}{r c r c }

\rotatebox{90}{\sf {\scriptsize{~~~~~~~negative log-likelihood $ \mathcal{L}(\cdot) $ } }}  & 
\includegraphics[width=0.45\textwidth,height=0.22\textheight,  trim =4.0cm 9.2cm 4.5cm 9.cm, clip = true ]{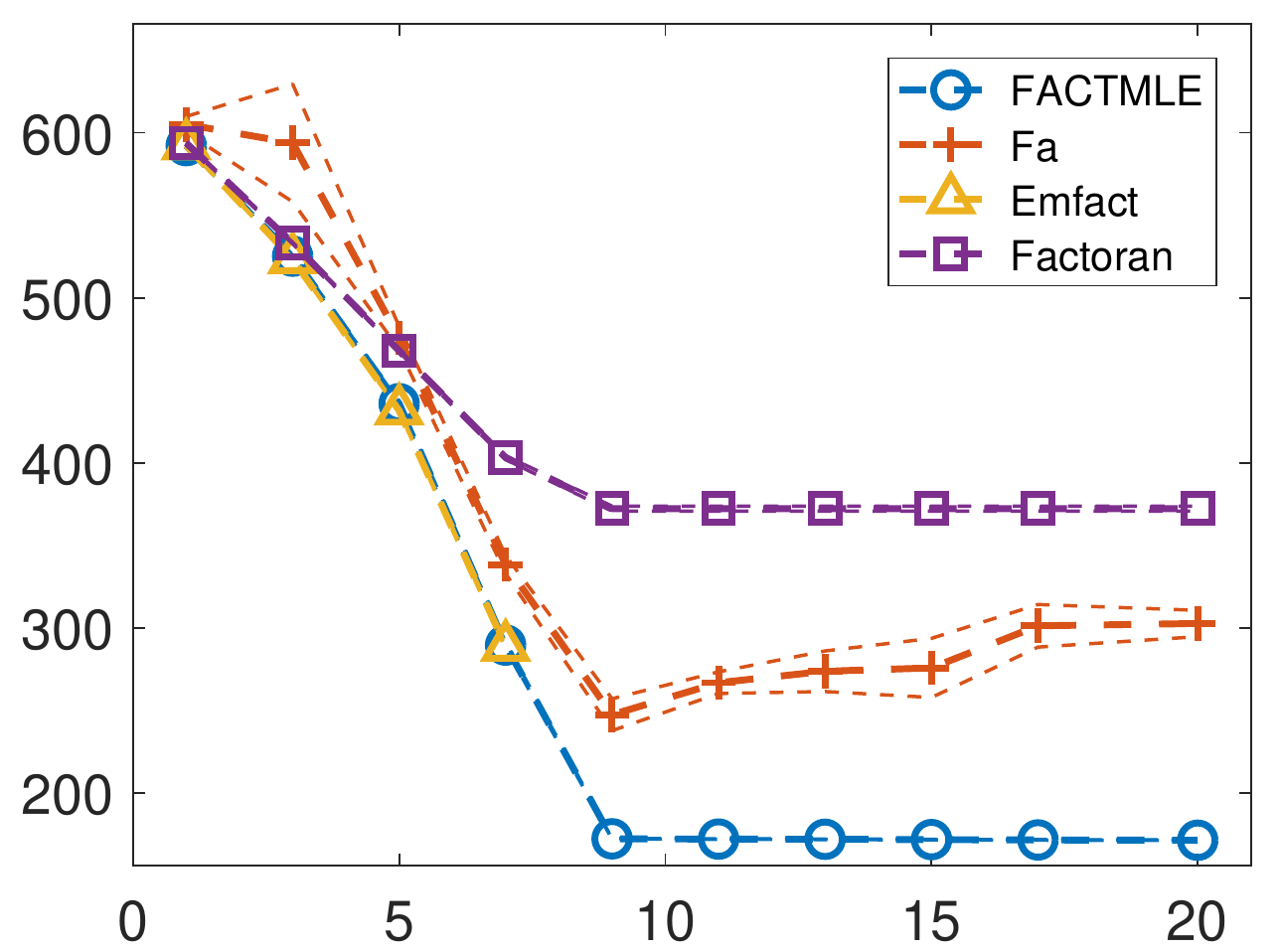} 

& \rotatebox{90}{\sf {\scriptsize{~~~~~~~~~~~~~~~Time in sec }}}  & \includegraphics[width=0.45\textwidth,height=0.22\textheight,  trim =4.2cm 9.2cm 4.5cm 9.cm, clip = true ]{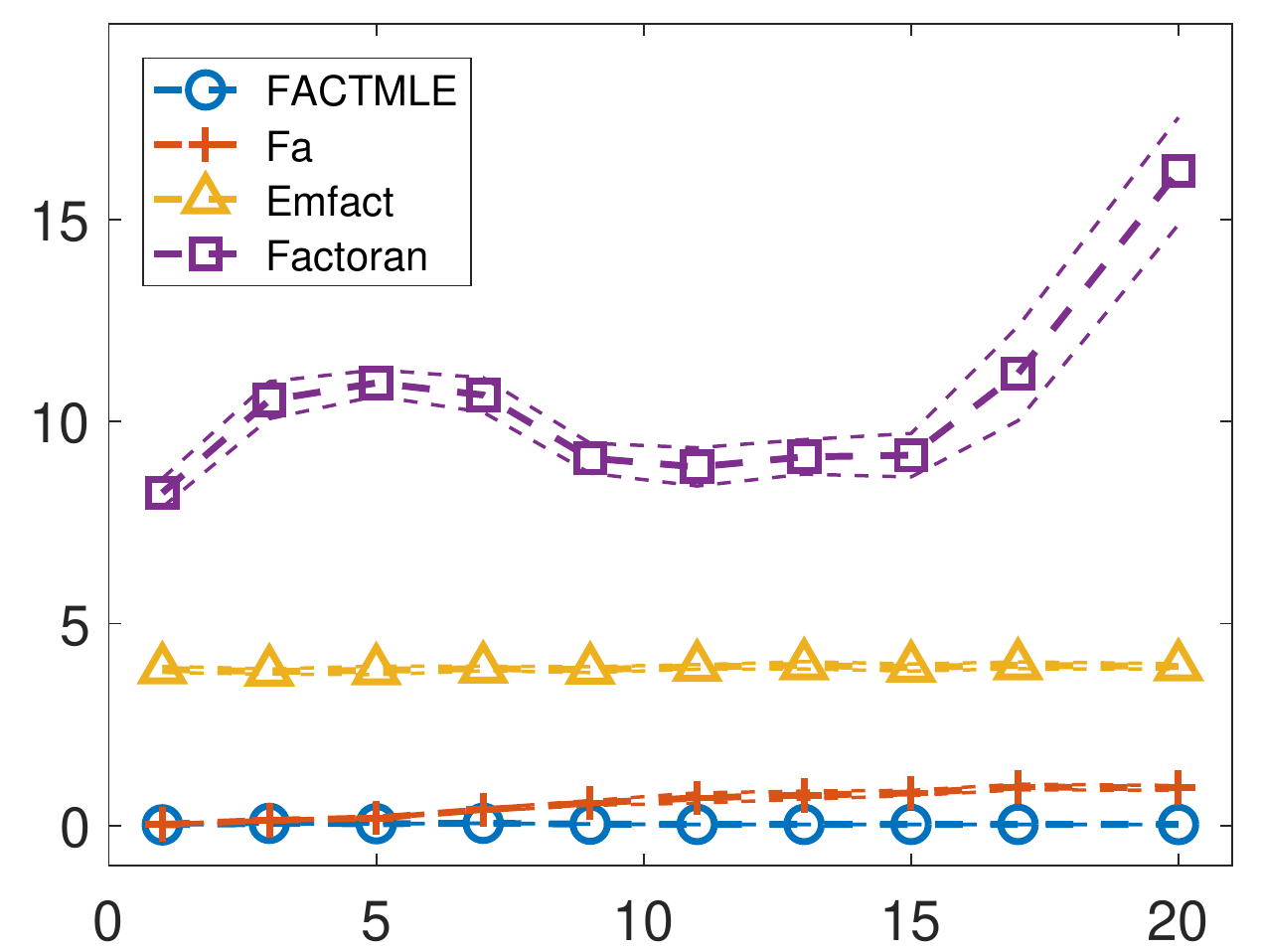} \\
& \sf {\scriptsize{Number of factors}} &  & \sf {\scriptsize{Number of factors}}  \smallskip \smallskip\\

& {\scriptsize{Zoomed Plot: \textsc{Factmle} vs {\texttt{Fa}} }} &  & {\scriptsize{Zoomed Plot: \textsc{Factmle} vs {\texttt{Fa}} }}   \\

 \rotatebox{90}{\sf {\scriptsize{~~~~~negative log-likelihood $ \mathcal{L}(\cdot) $ } }}  & \includegraphics[width=0.45\textwidth,height=0.22\textheight,  trim =4.0cm 9.2cm 4.5cm 9.cm, clip = true] {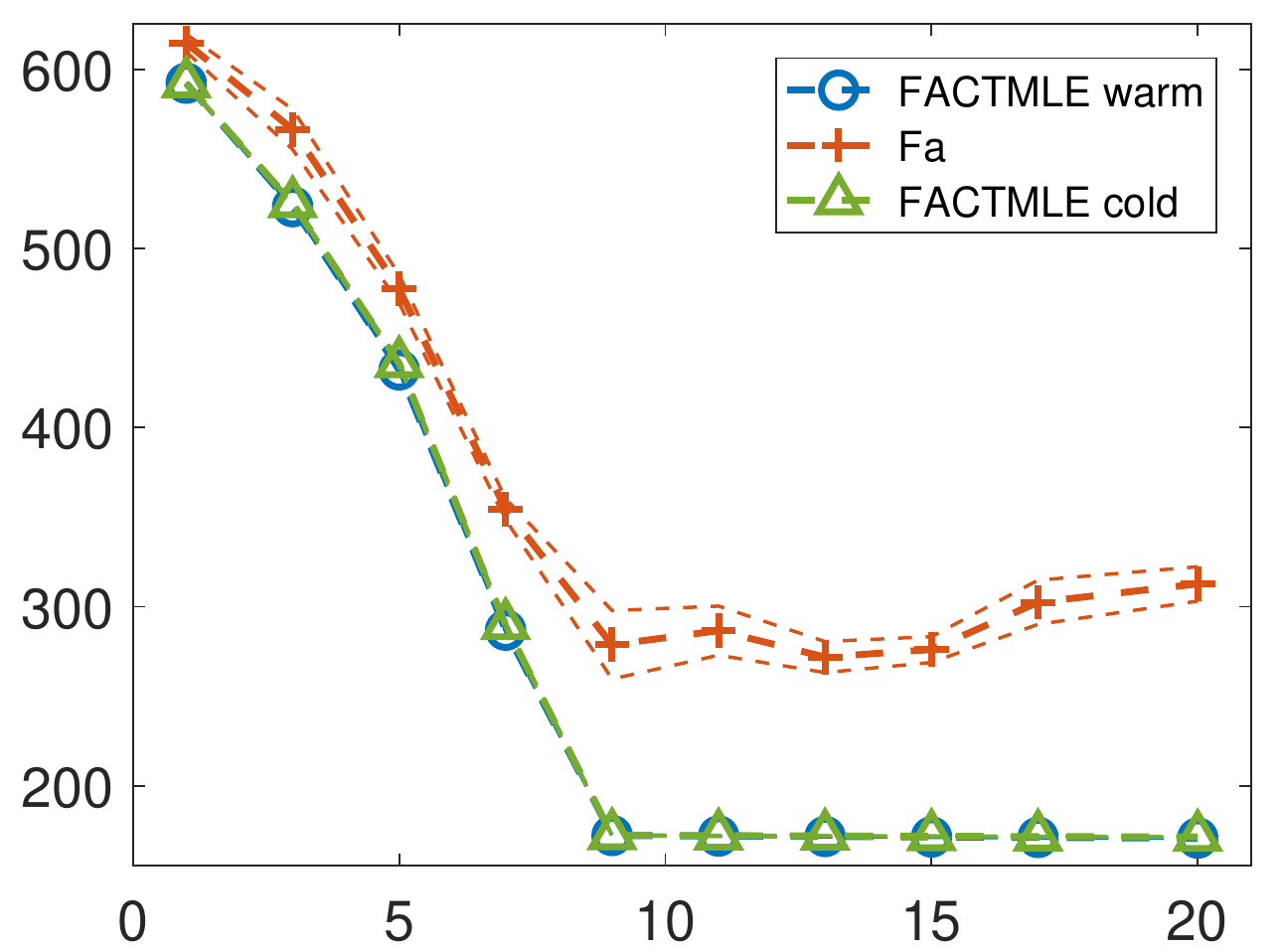} 

& \rotatebox{90}{\sf {\scriptsize{~~~~~~~~~~~~~~~Time in sec }}}  & \includegraphics[width=0.45\textwidth,height=0.22\textheight,  trim =4.2cm 9.2cm 4.5cm 9.cm, clip = true]{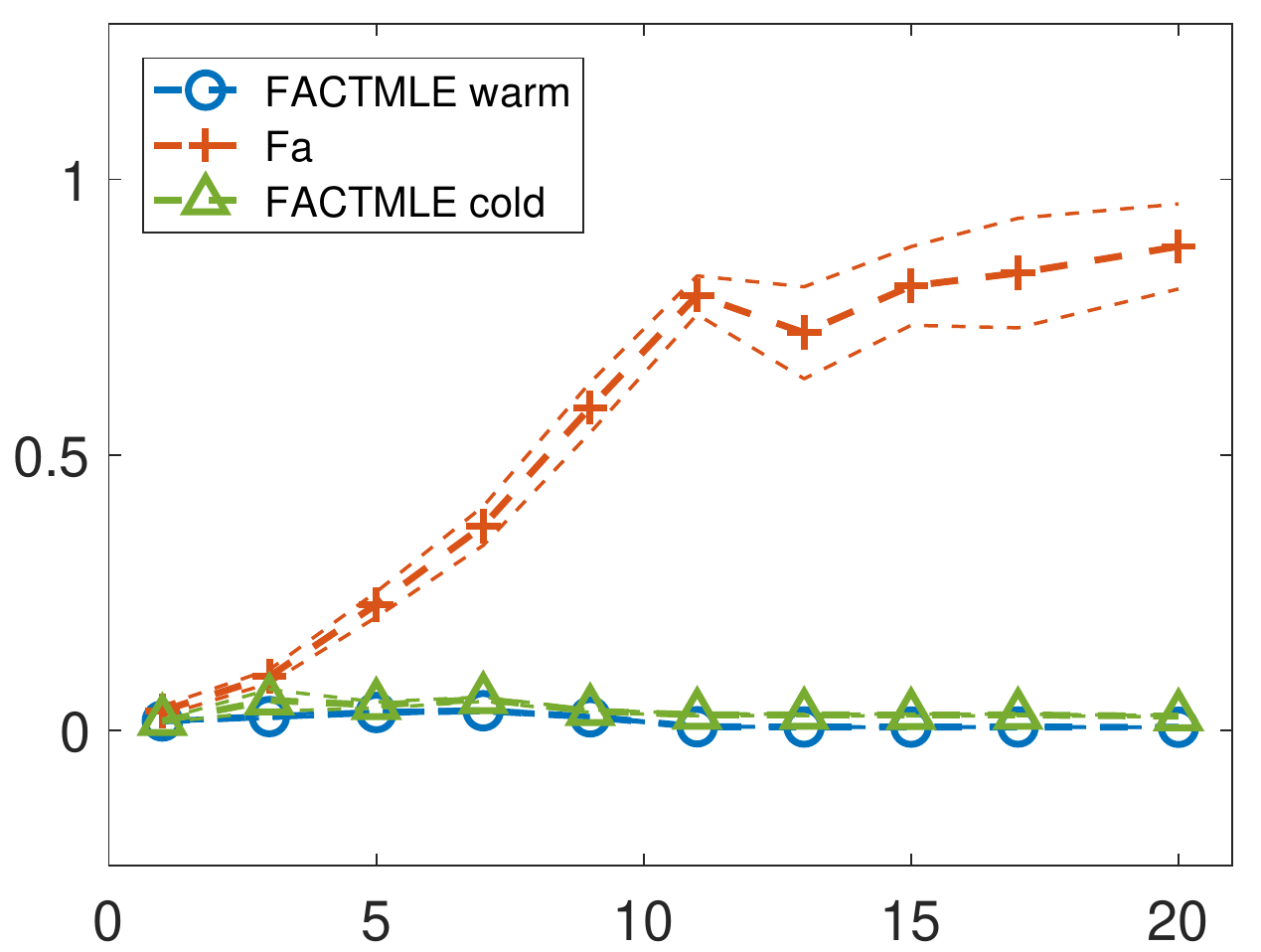} \\

 & \sf {\scriptsize{Number of factors}} &  & \sf {\scriptsize{Number of factors}}

\end{tabular}}
\caption{ {\small{ Figure showing performances of different ML methods for Factor Analysis for synthetic data with $(n,p) = (2200,200)$.  
In the two figures in first row, we compare
the performance of \texttt{Fa}, \textsc{Factmle}, \texttt{Factoran} and \texttt{Emfact} in terms of negative log-likelihood and run-time. We see that our method \texttt{Factmle}
outperforms all other methods, both in  terms of run-time and negative log-likelihood value. \texttt{Emfact} produces negative estimates of $ \B\Phi $ in some cases, and in those cases the result obtained from \texttt{Emfact} do not appear in the plots. The lower panel shows a zoomed-in version of \texttt{Fa} (the best competitor of \textsc{Factmle})
 versus \textsc{Factmle} with warm start and without warm-start (denoted as ``cold''). 
 %%The plots in second row are zoomed in version of plots in row one, with an added random initialization figure.
  We see that the performance of \textsc{Factmle} (with or without warm start) is better than that of \texttt{Fa}. \textsc{Factmle}
 with warm start gives slightly better run-time. Results are averaged over ten replications and the bands represent point-wise one standard error bars.  
  %%%$L_mean = 10, L_var = 1, exp_mean = 10, r_pop = 8$ 
  } } } \label{fig:Simulated Example a}
\end{figure}

\subsubsection{Further comparisons with \texttt{Fa}} 

%%%Figure~\ref{fig:Simulated Example a} showed that the performance of \texttt{Fa} is better \texttt{Factoran} and \texttt{Emfact}. 
In our experience, we systematically observed that among  \texttt{Fa}, \texttt{Factoran} and \texttt{Emfact};  \texttt{Fa} always 
emerged as the winner (see also Figure~\ref{fig:Simulated Example a}) in terms of numerical stability,  quality of solutions delivered and  run-times.
We also note that \texttt{Fa} is the only method among the three which applies for both $n>p$ and $n\leq p$.
Hence we perform a more detailed comparison between \texttt{Fa} and our proposed method: \textsc{Factmle}. 
We consider four different datasets in our experiments one synthetic and the other real -- this includes both the cases 
$n >p$ and $p >n$.  
For all the numerical examples in this section, we had the same convergence criterion for 
both \texttt{Fa} and \textsc{Factmle}.
We ran both \texttt{Fa} and \textsc{Factmle} for a maximum of 2000 iterations and tolerance threshold 
of $\eta=10^{-8}$ (as described in the display of Algorithm~1). If $f_k(\text{alg})$ denotes the objective value (negative log-likelihood) for method $\text{alg} \in \{ \texttt{Fa},\textsc{Factmle}\}$ at iteration $k$;
then we obtain the best objective value across all the methods and set it to $f_*$. We then study the first time at which 
an algorithm ``$\text{alg}$'' reaches a tolerance level of ``Tol'':
$(f_k(\text{alg}) - f_*)/|f_*| \leq \text{Tol}$.

We considered the following datasets in our experiments.
   
 \smallskip

%%%%%%%%%%%%%%%%%%%%% Put Table %%%%%%%%%%%%%%%

\noindent {\bf{Example a:}}  (Simulated Data) $(p>n)$: We consider 3 synthetic datasets where, 
the number of covariates $p$ is larger than the number of samples $n$; we took:
$(n,p)= (500,5000)$, $(n,p)= (50,10^4)$ and $(n,p)= (150,10^4)$.
%%%$ (n,p) = (50,10000) $. 
The data was simulated as per the setup mentioned in Section~\ref{subs:four_compare}, with 
the entries $\phi^0_{i}$ drawn iid from an exponential distribution with mean $ 1 $. We set $r^{0} = 5$.
We considered a sequence of solutions for the ML factor analysis problem for 15 equi-spaced values of 
$r \in [1, 18]$.
 
  \smallskip

\noindent {\bf{Example b:}} (Phoneme Data), $(n>p)$:
The data were extracted from the TIMIT database a widely
used resource for research in speech recognition. The data was downloaded from the companion 
website~\footnote{Available at:~\url{https://web.stanford.edu/~hastie/ElemStatLearn/datasets/phoneme.data}} of the textbook~\cite{FHT-09} -- it consists
of 4509 log-periodograms of length 256, i.e. in this example $ (n,p) = (4509,256)$. 
We considered a sequence of solutions for the ML factor analysis problem for 18 equi-spaced values of 
$r \in [1, 27]$. The sample covariance matrix $\M{S}$ was poorly conditioned. 
The condition number of $\M{S}$ was $\sim 3.9\times10^{3}$.
%%For Algorithm~1, we took a value of $\epsilon = 10^{-10}$ (in Problem~\eqref{obj-2-0-0}) 

 \smallskip

%A dataset\footnote{The dataset was downloaded from~\url{https://web.stanford.edu/~hastie/ElemStatLearn/datasets/phoneme.data} } was formed
%by selecting five phonemes for classification based on digitized speech
%from this database. From continuous
%speech of 50 male speakers, 4509 speech frames of 32 msec duration
%were selected, approximately 2 examples of each phoneme from each
%speaker. From each speech frame, we computed a log-periodogram, which is one
%of several widely used methods for casting speech data in a form suitable
%for speech recognition. Thus the dataset consist
%of 4509 log-periodograms of length 256, i.e. in this example $ (n,p) = (4509,256) $. 

%%%%% Put table %%%%%%%%%%%%%%%%%%%

%%%%%%%%%%%%%%%%%%%%%%% Put table %%%%%%%%%%%%%%%%

\noindent {\bf{Example c:}} (ZipCode Data), $(n>p)$:
%In our final real dataset example we compare the performance of \texttt{Fa} and \textsc{Factmle} on the Zip Dataset. 
This is the by now famous ZipCode dataset\footnote{Available at~\url{https://web.stanford.edu/~hastie/ElemStatLearn/datasets/zip.test.gz}} which was 
generated by scanning normalized handwritten digits by the U.S. Postal Service. These are $16 \times 16$ grayscale images corresponding to digits (0-9) that are normalized/deslanted.
%% -- they are $16 \times 16$ grayscale images. 
%images  here have been deslanted and size normalized, resulting
%in $16 \times 16$ grayscale images. The data are in two gzipped files, and each line consists of the digit
%id (0-9) followed by the 256 grayscale values. 
The images were vectorized, and we created a data matrix comprising of digits $0$ and $6$.
%%% resulting in a
%We have taken the subset with digit id $ 0 \& 6 $ and choose column 2 to 250
%of the 256 grey scale image value. 
The data matrix $\M{X}_{n \times p}$ had dimensions $(n,p) = (1858,249)$. 
We considered a sequence of solutions for the ML factor analysis problem for 15 equi-spaced values of 
$r \in [1, 17]$.
The condition number of $\M{S}$ here was $\sim 3.8 \times 10^{17}$.
%%For Algorithm~1, we took a value of $\epsilon = 10^{-10}$ (in Problem~\eqref{obj-2-0-0}) 

%%%%%%%%%%%%%%%%

 \smallskip

\noindent {\bf{Example d:}} (Cancer Data), $(p>n)$:
%%%In the second real data example we applied \textsc{Factmle} and \texttt{Fa} on a 
This is a high-dimensional microarray dataset available from\footnote{Data available at \url{http://statweb.stanford.edu/~tibs/ElemStatLearn/datasets/14cancer.info}}
 the companion website of~\cite{FHT-09}. 
It consists of gene expression measurements for $p=16063$ genes from $n=144$ individuals.
%%%% i.e. in this example $ (n,p) = (144,16063) $. 
We considered a sequence of solutions for the ML factor analysis problem for 15 equi-spaced values of 
$r \in [1, 22]$.
The condition number of $\M{S}$ here was $\sim 10^{20}$.
%For Algorithm~1, we took a value of $\epsilon = 10^{-10}$ (in Problem~\eqref{obj-2-0-0}) 

\smallskip

For Examples~b,c,d, the condition number of $\M{S}$ was quite high --- we took $\epsilon = 10^{-10}$ (in Problem~\eqref{obj-2-0-0}) and Algorithm~1 was provided with this choice of $\epsilon$. For Algorithm~1, 
we did not use the continuation strategy as described in Section~\ref{existance-of-ML}. For the synthetic dataset (Example~a), 
we took $\epsilon = 10^{-3}$. In Algorithm~1, the diagonal entries of $\B\Phi$ were drawn from a uniform $[0,1]$ distribution for the case $r=1$; and solutions for other values of 
$r$ were obtained via warm-starts. For {\texttt{Fa}} we used its default initialization scheme since it does not allow for warm-starts.

For all cases with $n<p$, the low rank SVD step in Algorithm~1 was performed according to the description given 
in Section~\ref{sec:Computational details} (for the case $p\gg n$). For $n>p$, we used the low rank SVD method of 
Section~\ref{sec:Computational details} (for the case $n > p$).

Table~\ref{tab:simus-rest} shows the results for all the four examples. We observe that \textsc{Factmle} clearly works extremely well in terms of obtaining a good objective value in much smaller run-times, compared
to {\texttt{Fa}}. What is most important however, is that  \textsc{Factmle}  is numerically robust --- in fact much more stable and reliable than {\texttt{Fa}} which often encounters problems with convergence.
In many cases, as soon as $r$ is not too small, {\texttt{Fa}} is found to be trapped in poor fixed points with suboptimal objective values. These are referred to by the moniker ``>*'' in Table~\ref{tab:simus-rest}. 
We note that these datasets are challenging as the condition numbers associated with them are extremely large --- however, we do expect to encounter datasets of this form often in real-life settings, thus having a good robust algorithm is of paramount importance. Our proposed algorithm  \textsc{Factmle} does seem to meet this expectation. 

% and
%we considered ten replications. 
% The entries of the matrix $ \B{L^0} $ are iid and are generated from a Gaussian distribution with mean $ \mu =10 $ and $ \sigma^2 =1 $. The rank of $ \B{L^0} $ was set to be $ r^{pop} = 5 $. 

%%% Put table %%%%%%%%%%%%%%%%%%%%%%%%%%%%      

\begin{table}[h!]
\begin{tabular}{ c c }
\begin{tabular}{|c | c | c|}  %%~~~~~~~~~~~~~~~~~~~~                                                                                 
\multicolumn{3}{c}{Synthetic Data $(n=500,p=5000)$} \\ 
\multicolumn{3}{c}{Time (in secs)}  \\
\hline 
Tol   & {\texttt{Fa}} &  {\textsc{Factmle}}   \\                                      
\hline                                                                                                            
$ 10^{-2} $ &>*83.142    (4.174)& 4.497    (0.041)   \\                                                   
\hline                                                                                                            
$ 10^{-3} $ &   >*107.713   (2.148)& 7.702    (2.508)\\                                              
\hline                                                                                                            
$ 10^{-4} $ & >*112.145    (1.157)& 18.490    (5.706)\\                                               
\hline                                                                                                            
$ 10^{-5} $   & >*112.146    (1.157)& 29.369    (6.154)   \\   \hline                                                                                                            
\end{tabular} %%~~~~~~~~~~~~~~~~~~~~
&
\begin{tabular}{|c|c|c|}  %%~~~~~~~~~~~~~~~~~~~~                                                                                
  \multicolumn{3}{c}{Synthetic Data $(n=150,p=10^4)$}\\ 
  \multicolumn{3}{c}{Time (in secs)}  \\ \hline  
Tol   & {\texttt{Fa}} &  {\textsc{Factmle}}   \\                                      
\hline                                                                                                            
$ 10^{-2} $ &  >*48.489    (4.605)& 1.680    (0.027)  \\                                                   
\hline                                                                                                            
$ 10^{-3} $ &  >*90.445    (1.601)&1.920    (0.048)   \\                                                   
\hline                                                                                                            
$ 10^{-4} $ &   >*93.004    (0.774)&   4.393    (0.592)   \\                                                  
\hline                                                                                                            
$ 10^{-5} $ &  >*93.004    (0.773)&     6.662    (0.764)  \\   \hline                                                                                                            
\end{tabular} %%~~~~~~~~~~~~~~~~~~~~
\\~\\
\begin{tabular}{|c|c|c|}      %%~~~~~~~~~~~~~~~~~~~~                                                                                             
 \multicolumn{3}{c}{Synthetic Data $(n=50,p=10^4)$}\\ 
 \multicolumn{3}{c}{Time (in secs)}  \\ \hline 
 Tol  & {\texttt{Fa}} &  {\textsc{Factmle}}   \\                                      
\hline                                                                                                            
$ 10^{-2} $ & >*43.794 ~(1.160) & 0.558 ~(0.097)   \\                                                   
\hline                                                                                                            
$ 10^{-3} $ & >*50.798 ~(0.712) & 0.882 ~(0.228)   \\                                                   
\hline                                                                                                            
$ 10^{-4} $ & >*52.669 ~(0.362) & 1.901 ~(0.375)   \\                                                  
\hline                                                                                                            
$ 10^{-5} $ & >*52.959 ~(0.296) & 3.798 ~(0.496)   \\   \hline                                                                                                            
\end{tabular} %%~~~~~~~~~~~~~~~~~~~~ 
&
\begin{tabular}{|c|c|c|} %~~~~~~~~~~~~~~~~~~~~                                                                    
 \multicolumn{3}{c}{Phoneme Data $(n,p) = (4509,256)$}\\ 
 \multicolumn{3}{c}{Time (in secs)}  \\ \hline 
 Tol & {\texttt{Fa}} &  {\textsc{Factmle}}  \\ \hline                                                                     
$ 10^{-2} $ & 3.038 ~(0.385)& 0.125 ~(0.006)  \\             
\hline                                                                     
$ 10^{-3} $ & >*23.902~(1.624)& 0.187 ~(0.007)  \\            
\hline                                                                     
$ 10^{-4} $ & >*29.635 ~(1.430)& 0.420 ~(0.007)  \\            
\hline                                                                     
$ 10^{-5} $ & >*37.035 ~(1.245)& 0.669 ~(0.008) \\            
\hline                                                                     
\end{tabular}%~~~~~~~~~~~~~~~~~~~~
\\~\\
\begin{tabular}{|c|c|c|}    %%~~~~~~~~~~~~~~~~~~~~                                       
            \multicolumn{3}{c}{ZipCode Data $(n,p) = (1858,249)$}\\
            \multicolumn{3}{c}{Time (in secs)}  \\ \hline 
Tol  & {\texttt{Fa}} &  {\textsc{Factmle}}   \\
\hline                                                                     
$ 10^{-2} $ & 1.558 ~(0.060) & 0.111 ~(0.007)   \\             
\hline                                                                     
$ 10^{-3} $ & 5.822 ~(0.099) & 0.171 ~(0.008)   \\             
\hline                                                                     
$ 10^{-4} $ & >*8.673 ~(0.100) & 0.445 ~(0.010)   \\             
\hline                                                                     
$ 10^{-5} $ & >*10.803 ~(0.099) & 0.640 ~(0.012)     \\            
\hline                                                                     
\end{tabular} %%~~~~~~~~~~~~~~~~~~~~   
&
\begin{tabular}{|c|c|c|}    %%~~~~~~~~~~~~~~~~~~~~                                         
  \multicolumn{3}{c}{Cancer Data $(n=144,p=16063)$}\\
  \multicolumn{3}{c}{Time (in secs)}  \\ \hline 
 Tol                        & {\texttt{Fa}} &  {\textsc{Factmle}}   \\
\hline                                                                     
$ 10^{-2} $ & 1.055 ~(0.016) & 3.816 ~(0.014)   \\         
\hline                                                                     
$ 10^{-3} $ & 12.021 ~(0.099) & 5.346 ~(0.040)    \\        
\hline                                                                     
$ 10^{-4} $ & >*81.979 ~(0.355) & 8.516 ~(0.050)  \\         
\hline                                                                     
$ 10^{-5} $ & >*94.586 ~(0.357) & 12.431 ~(0.060)    \\     \hline                                                                     
\end{tabular}%%~~~~~~~~~~~~~~~~~~~~ 
\end{tabular}
\caption{\small{Performance of \textsc{Factmle} and \texttt{Fa} for different real and synthetic datasets; as described in the text.  We show the times (secs) taken by different algorithms to compute the entire path of 
solutions for different values of $r$, as specified in the text. 
In all the above examples, a symbol ``>*'' means that the corresponding 
algorithm did not converge to the specified tolerance level for multiple values of $r$ and replications. This usually does not happen when $r$ is small (say less than 5), but is indeed a common problem whenever $r$ becomes larger.  \textsc{Factmle} is seen to be a clear winner across all instances. The results are averaged over 10 replications with standard errors in parenthesis. }}\label{tab:simus-rest}
\end{table}

%   >*48.489    (4.605)& 1.680    (0.027)
%   >*90.445    (1.601)&1.920    (0.048)
%   93.004    (0.774)&   4.393    (0.592)
%   93.004    (0.773)&     6.662    (0.764)

% >*83.142    (4.174)& 4.497    (0.041) \\
%   >*107.7132   ( 2.148)& 7.702    (2.508)\\
%   >*112.1456    (1.157)& 18.490    (5.706)\\
%   >*112.1461    (1.157)& 29.369    (6.154) \\

\section{Generalizing beyond a diagonal $\B\Phi$}\label{gen-err-structure}
The most conventional version of FA concerns a diagonal matrix and the focus of the paper is on this case.
However, we show that our proposed framework can also naturally adapt to more general constraints on $\B\Phi$.
Towards this end, we have the following remark:
\begin{rem}\label{gen-Phi}
For any $r \in [p]$, it follows from the proof of Proposition~\ref{lem-diff-convex1}, that  for any $\B\Phi \succeq \M{0}$ (not necessarily diagonal) the function 
$$\B\Phi \mapsto F_{2}(\B\Phi):=\sum_{i=1}^{r} \left ( \log\left(\R{max}\{1,\lambda_i^{*}\}\right) -\R{max}\{1,\lambda_i^{*}\} +1 \right),$$ 
where, $\{\lambda_i^{*}\}_{1}^{p}$ are the eigenvalues of $\M{S} ^{\tfrac{1}{2}} \B\Phi \M{S}^{\tfrac{1}{2}}$ is concave in $\B\Phi \succeq \M{0}$.
\end{rem}
For a general $\B\Psi \succ \M{0}$, the maximum likelihood FA problem:
\begin{equation}\label{obj-2-0-0-gen}
\begin{aligned}
\mini~~~ - \log\det(\B\Sigma^{-1}) + \R{tr}(\M\Sigma^{-1}\M{S}) ~~~\sbt~~\M{\Sigma} = \M{\Psi} + \M{L} \M{L}^\top,    \B{\Psi} \succ \M{0}, 
\end{aligned}
\end{equation}
with optimization variables $\B\Sigma, \M{L}, \B\Psi$ (with same dimensions as before)
is equivalent to 
\begin{equation}\label{obj-2-0-0-gen-1}
\mini~~~H(\B\Phi):= F_{1}(\B\Phi) - F_{2}(\B\Phi) ~~~ \sbt ~~~~ \B{\Phi} \succ \M{0}, 
\end{equation}
where, $F_{1}(\B\Phi) := - \log\det(\B\Phi) + \R{tr}(\B\Phi\M{S})$; and $F_{2}(\B\Phi)$ is as defined above (Remark~\ref{gen-Phi}). 
Note that $F_{1}(\B\Phi)$ and $F_{2}(\B\Phi)$ are convex in $\B\Phi$. However, in the context of FA it is meaningful to consider 
additional regularization on $\B\Phi$. For example, $\B\Phi$ can be block-diagonal; banded or have entries with small $\ell_{1}$-norm (aka $\ell_{1}$-sparsity).

In presence of additional convex constraints on $\B\Phi$, i.e., 
say, $\B\Phi \in \texttt{X}$ the difference 
of convex optimization procedure described in Section~\ref{sec-algo} readily applies. Algorithm~1 gets modified to the following convex problem with semidefinite constraints
\begin{equation}\label{update-gen-case-1}
\B\Phi^{(k+1)} \in \argmin \{   F_{1} (\B\Phi) - \langle \partial F_{2} (\B\Phi^{(k)}), \B\Phi \rangle : \B\Phi \in \texttt{X}, \B\Phi \succ \M{0}\}. 
\end{equation}
Clearly, the efficiency of this procedure depends upon how efficiently the subproblem~\eqref{update-gen-case-1} can be solved -- this depends
upon $\texttt{X}$. If $\B\Phi$ is banded, block diagonal, or has small $\ell_{1}$-norm, $\texttt{X}$ can be described by a polyhedral set. 

A special structure of $\B\Phi$ is when it is block diagonal $\B\Phi = \text{Blk}\diag(\B\Phi_{1}, \ldots, \B\Phi_{m})$ where, 
$\B\Phi_{i} \in \Re^{p_{i} \times p_{i}}$ is the $i$th block, for $i =1, \ldots, m$. 
If $\texttt{X} = \{ \B\Phi: \B\Phi \succeq \M{0}\}$, then the optimization problem~\eqref{update-gen-case-1} splits into 
$m$ different blocks; and the solution for the $i$th block is given by:
$ \B\Phi_{i}^{(k+1)} = (\M{S}_{ii} - \B\nabla_{k,ii})^{-1},$
where, $\M{S}_{ii}$ (and $\B\nabla_{k,ii}$) are the $i$th blocks of $\M{S}$ (respectively, $\B\nabla_{k}$):
\begin{equation}
\M{S} =\left[\begin{array}{cccc}
\M{S}_{11}&\M{S}_{12}&\cdots &\M{S}_{1m}\\
\vdots & &\ddots &\vdots \\
\M{S}_{m1}&\M{S}_{m2}&\cdots &\M{S}_{mm}\\
\end{array}\right] ~~\text{and}~~ 
\B{\nabla}_{k}= \left[\begin{array}{cccc}
\B{\nabla}_{k,11}&\B{\nabla}_{k,12}&\cdots &\B{\nabla}_{k,1m}\\
%%%\B{\nabla}_{k,21}&\B{\nabla}_{k,22}&\cdots &\B{\nabla}_{k,2m}\\
\vdots & &\ddots &\vdots \\
\B{\nabla}_{k,m1}&\B{\nabla}_{k,m2}&\cdots &\B{\nabla}_{k,mm}\\
\end{array}\right]
\end{equation}
where,  $\M{S}_{ij}$ (and $\B{\nabla}_{ij}$) is $p_i \times p_j$ submatrix of $\M{S}$ (respectively, $\B{\nabla}_{k}$). 
Note that $\B\nabla_{k}$ is a subgradient of $F_{2}(\B\Phi)$ w.r.t. $\B\Phi$ evaluated at $\B\Phi^{(k)}$.

One can also consider an $\ell_{1}$-norm sparsity on the entries of $\B\Phi$, akin to graphical lasso~\cite{FHT2007a,BGA2008} commonly used in learning sparse inverse covariance matrices. 
In this case, update~\eqref{update-gen-case-1} leads to an optimization problem of the form:
$$ \mini ~~ -\log\det(\B\Phi) + \langle \M{S} - \nabla_{k}, \B\Phi \rangle + \lambda \sum_{ij} |\phi_{ij}| ~~~\sbt ~~~ \B\Phi \succeq \M{0},~~~ \phi_{ij} \geq 0~~\forall i,j; $$
which can be solved using adaptions of standard algorithms for the graphical lasso -- see~\cite{FHT2007a,FHT-09,BGA2008}.

\begin{appendix}
\section{Appendix} \label{appendix}

\begin{prop}\label{prop-contty-1}(See Section~6 in~\cite{borwein-conv})
Suppose the function $g: \M{E} \mapsto (-\infty, \infty)$ is convex, and the point 
$x$ lies in interior of dom(g) with $\M{E} \subset \Re^{m}$. If $x^{r} \rightarrow x$ (where, $x^r,x \in \M{E}$) and $\nu^{r} \rightarrow \nu$; as $r \rightarrow \infty$, where, $\nu^r$ is a subgradient of $g$ evaluated at 
$x^r$. Then $\nu$ is a subgradient of 
$g$ evaluated at $x$.
\end{prop}

\subsection{Proof of Proposition~\ref{prop:All limit points are stationary point }}
\label{prop:proof of all limit points are stationary point}

Note that the objective function $f(\B\phi)$ (see~\eqref{obj-2}) is unbounded above when 
$\phi_{i} \rightarrow 0$ for any $i \in [p]$ -- see also Proposition~\ref{bdd-away-from-zero}. This implies that there exists a $\alpha>0$ such that 
$\p{(k)} \in [\alpha, \tfrac1\epsilon]^p$ for all $k$ (sufficiently large).
The boundedness of $\p{(k)}$ implies the existence of a limit point of $\p{(k)}$, say, $\p{*}$.
%%%Let $ \B\phi^* $ be a limit point of $ \left\lbrace \p{k} \right\rbrace_{ k \geq 0 } $,  i.e. 
Let $\p{(k_j)}$ be a subsequence such that $ \p{(k_j)} \rightarrow \p{*}$ as $j \rightarrow \infty$. 
Note that for every $k$, $\B\phi^{(k+1)} \in \argmin_{\B\phi \in \mathcal C} F(\B\phi; \B\phi^{(k)})$ is equivalent to
\begin{equation}\label{subseq-11}
\left\langle \nabla f_{1}(\B\phi^{(k+1)}) - \partial f_{2} ( \B\phi^{(k)}), \B\phi - \B\phi^{(k+1)} \right\rangle \geq 0~~\forall 
\B\phi \in {\mathcal C}.
\end{equation}
Now consider the sequence $\B\phi^{(k_j)}$ as $j\rightarrow \infty$. Using the fact that 
$\B\phi^{(k+1)} - \B\phi^{(k)} \rightarrow \M{0}$; it follows from the continuity of $\nabla f_{1}(\cdot)$ 
that:  $\nabla f_{1}(\B\phi^{(k_j+1)}) \rightarrow  \nabla f_{1}(\B\phi^{*})$ as $j \rightarrow \infty$.

Note that $ \partial f_2(\B\phi^{(k_j)})$ (see~\eqref{subgrad-exp}) is bounded as 
$\p{(k_j)} \in [\alpha, \tfrac1\epsilon]^p$. Passing onto a further subsequence $\{k'_j\}$ if necessary, 
it follows that $\partial f_2(\B\phi^{(k'_j)}) \rightarrow  \vartheta$. Using Proposition~\ref{prop-contty-1},
we conclude that $\vartheta$ is a subgradient of $f_2$ evaluated at $\B\phi^*$. 
As $k'_j \rightarrow \infty$, the above argument along with~\eqref{subseq-11} implies that:
\begin{equation}\label{subseq-12}
\left\langle \nabla f_{1}(\B\phi^{*}) - \partial f_{2} ( \B\phi^{*}), \B\phi - \B\phi^{*} \right\rangle \geq 0~~\forall 
\B\phi \in {\mathcal C}, 
\end{equation}
 where, $\partial f_{2} ( \B\phi^{*})$ is a subgradient of $f_{2}$ evaluated at $\B\phi^{*}$.
 \eqref{subseq-12} implies that $\B\phi^*$ is a first order stationary point. 
\qed

\end{appendix}
\bibliographystyle{chicago}
%\bibliography{../LM\lambda_paper/rahul_dbm.bib}
%\bibliography{../LM\lambda_paper/rahul_dbm.bib}
\bibliography{paper1_dbm,rahul_dbm}
\end{document}